\documentclass[11pt]{amsart}
\usepackage{amsmath,amsfonts,amssymb,mathrsfs}
\usepackage{color}
\usepackage{graphicx}
\usepackage{tikz}
\usepackage{array}
\usepackage{subfigure}
\usepackage{epstopdf}
\usepackage{rotating}
\newcommand{\ignore}[1]{}

\title[3D Vlasov equations
of non-Newtonian fluids type] {Well-posedness of strong solutions for the Vlasov equation
coupled to non-Newtonian fluids in dimension three}

\author[Kyungkeun Kang]{Kyungkeun Kang}
\address[Kyungkeun Kang]{\newline Department of Mathematics \newline Yonsei University, Seoul 03722, Korea (Republic of)}
\email{kkang@yonsei.ac.kr}
\author[Hwa Kil Kim]{Hwa Kil Kim}
\address[Hwa Kil Kim]{\newline Department of Mathematics Education \newline Hannam University, Daejeon 34430, Korea (Republic of)}
\email{hwakil@hnu.kr}
\author[J.-M. Kim]{Jae-Myoung Kim}
\address[Jae-Myoung Kim]{\newline
Department of Mathematics Education, Andong National University
\newline Andong, Gyeongsangbuk-do, 36729, Korea (Republic of)}\email{jmkim02@anu.ac.kr}

\begin{document}

\newtheorem{theorem}{Theorem}[section]
\newtheorem{lemma}[theorem]{Lemma}
\newtheorem{corollary}[theorem]{Corollary}
\newtheorem{proposition}[theorem]{Proposition}
\newtheorem{remark}[theorem]{Remark}
\newtheorem{definition}[theorem]{Definition}
\newtheorem{example}[theorem]{Example}
\newtheorem{assumption}[theorem]{Assumption}

\newcommand{\bke}[1]{\left( #1 \right)}
\newcommand{\bkt}[1]{\left[ #1 \right]}
\newcommand{\bkv}[1]{\left| #1 \right|}
\newcommand{\bket}[1]{\left\{ #1 \right\}}

\renewcommand{\theequation}{\thesection.\arabic{equation}}
\renewcommand{\thetheorem}{\thesection.\arabic{theorem}}
\renewcommand{\thelemma}{\thesection.\arabic{lemma}}
\newcommand{\bbr}{\mathbb R}
\newcommand{\R}{{\mathbb R }}
\newcommand{\Z}{{\mathbb Z }}
\newcommand{\bbe}{\mathbb E}
\newcommand{\abs}[1]{\left| #1 \right|}
\newcommand{\bbz}{\mathbb Z}
\newcommand{\bbn}{\mathbb N}
\newcommand{\bbs}{\mathbb S}
\newcommand{\bbp}{\mathbb P}
\newcommand{\bbt}{\mathbb T}
\newcommand{\norm}[1]{\left\Vert #1 \right\Vert}
\newcommand{\Ass}{\textup{(\textbf{A})\,}}
\newcommand{\ddiv}{\textrm{div}}
\newcommand{\bn}{\bf n}
\newcommand{\rr}[1]{\rho_{{#1}}}
\newcommand{\rabs}[1]{\left. #1 \right|}
\newcommand{\thh}{\theta}
\def\charf {\mbox{{\text 1}\kern-.24em {\text l}}}
\renewcommand{\arraystretch}{1.5}

\newenvironment{thm1}{{\par\noindent\bf
            Proof of Theorem \ref{thm1}. }}{\hfill\fbox{}\par\vspace{.2cm}}
\newenvironment{thm2}{{\par\noindent\bf
            Proof of Theorem \ref{thm2}. }}{\hfill\fbox{}\par\vspace{.2cm}}

\newenvironment{thm3}{{\par\noindent\bf
            Proof of Theorem \ref{thm3}. }}{\hfill\fbox{}\par\vspace{.2cm}}

\newenvironment{lem3}{{\par\noindent\bf
            Proof of Lemma \ref{r-300}. }}{\hfill\fbox{}\par\vspace{.2cm}}
\newenvironment{lem4}{{\par\noindent\bf
            Proof of Lemma \ref{r-400}. }}{\hfill\fbox{}\par\vspace{.2cm}}

\newenvironment{lem5}{{\par\noindent\bf
            Proof of Lemma \ref{local-ex-u}. }}{\hfill\fbox{}\par\vspace{.2cm}}

\thanks{}



\begin{abstract}
We consider the Cauchy problem for coupled system of Vlasov and
non-Newtonian fluid equations. We establish local well--posedness of
the strong solutions, provided that the initial data are regular
enough.
\end{abstract}

\maketitle

\section{Introduction}\label{sect:intro}
\setcounter{equation}{0}

We study the following coupled system of Vlasov and non-Newtonian
fluid equations in phase space $\R^3\times \R^3$:
\begin{align}
\begin{aligned} \label{main}
&\partial_t f +v \cdot \nabla_xf+\nabla_v \cdot [(u-v)f]=0,  \\
&\partial_t u -\nabla \cdot\big( G[|Du|^2] Du \big)+(u\cdot \nabla)u+\nabla p=-\int_{\R^3}(u-v)fdv, \\
&\mbox{div}\ u=0,
\end{aligned}
\end{align}
where $Du$ is the symmetric part of the velocity gradient, namely,
\begin{equation*}\label{vis}
~~Du = D_{ij}u:= \frac{1}{2}\Big(\frac{\partial u_i}{\partial
x_j}+\frac{\partial u_j}{\partial x_i}\Big), \ i,j=1,2,3.
\end{equation*}
In \eqref{main}, we indicate by
$u:\R^3\times (0,\, T)\rightarrow\R^3$, $p:\R^3\times (0,\,
T)\rightarrow\R$ and $f:\R^3\times\R^3\times (0,\, T)\rightarrow\R$
 the flow velocity vector, the scalar pressure and the density
function of particles, respectively.
In this article, we study the Cauchy problem of
\eqref{main} with
\begin{equation}\label{initial}
 f(x,v,0) = f_0(x,v), \quad u(x,0)= u_0(x),
\quad x,v \in \bbr^3,
\end{equation} and in particular, $u_0$ holds the compatibility condition,
that is, $\mbox{div}\ u_0=0$.

Here, we make some assumptions on the viscous part of the stress
tensor, $G[|Du|^2]$.
\begin{assumption}\label{G-assume}
We suppose that $G : [0,\infty) \rightarrow [0,\infty)$ is a smooth
function and satisfies the following the structure conditions: There
is a positive constant $m_0$ such that
 for any $s\in[0,\infty)$
\begin{equation}\label{G : property}
\begin{aligned}
&G[s] \geq m_0 , \qquad G[s]+ 2 G^{'}[s]s \geq m_0,\\
&|G^{(k)}[s] s^\alpha| \leq C_k  |G^{(k-1)}[s]|, \quad \alpha \in \{
0,1 \},
\end{aligned}
\end{equation}
where $G^{(k)}[\cdot]$ is the $k-$th derivative of $G$, and $m_0$
and $C_k$ are positive constants.
\end{assumption}

We remark that some typical types of $G[s]$ satisfying Assumption \ref{G-assume} are of the following power-law
models, e.g.
\begin{equation}\label{example-100}
G[s]=(m_0^{\frac{2}{q-2}}+s)^{\frac{q-2}{2}},\quad 2<q<\infty,
\end{equation} or
\[
G[s]=m_0+(\sigma+s)^{\frac{q-2}{2}},\quad 1<q<\infty,\quad \sigma>0.
\]
We recall some known results for the case $G[s]=1$, namely, the
fluid equations is of Newtonian case. Hamdache \cite{Ham} proved the
global existence of weak solutions to the time dependent Stokes
system coupled with the Vlasov equation in a bounded domain. Later,
existence of weak solution was extended to the Vlasov-Navier-Stokes
system by Boudin et al. \cite{B-D09} in a periodic domain (refer
also to \cite{G-J1, G-J2} for hydrodynamic limit problems). When the
fluid is inviscid, Baranger and Desvillettes established the local
existence of solutions to the compressible Vlasov-Euler equations
\cite{B-D}.

In case of non-Newtonian fluid, recently, Mucha et. al. \cite{MPP18}
investigated the Cucker--Smale flocking model coupled with an
incompressible viscous generalized Navier-Stokes equations with $G$
given in \eqref{example-100}, for $q\geq \frac{11}{5}$ in a periodic
spatial
domain $\bbt^3$. To be more precise, 
the following
equations are considered:
\begin{equation}\label{flocking-f}
\partial_t f +v \cdot \nabla_xf+\nabla_v \cdot [\bke{F_{CS}(f)+(u-v)}f]=0,
\end{equation}
\begin{equation}\label{flocking-u}
\partial_t u -\nabla \cdot\bke{G[|Du|^2]Du}+(u\cdot \nabla)u+\nabla p=-\int_{\R^3}(u-v)fdv,
\end{equation}
where
\[
F_{CS}(f)(t, x, v) = \iint_{\bbt^3\times \bbr^3} (w- v)\psi(|x-
y|) f (t, y, w)\,dydw
\]
and $\psi(\cdot)$ the communication weight is positive, decreasing
and $\|\psi\|_{C^1}< \infty$.

The viscosity part of stress tensor $G$ in \cite{MPP18} was assumed
to satisfy the structure conditions, which are given as
        \begin{equation} \label{Coer-D}
        G[s]s\geq C(\mu_0s+s^{\frac{q}{2}}), \quad \bkv{G[s]}\leq
        C(\mu_0+s)^{\frac{q-2}{2}},\quad  \mu_0>0,
        \end{equation}
        \begin{equation}\label{monotonicity}
        \bke{G[|Dv|^2]Dv-G[|Dw|^2]Dw: Dv-Dw} \geq
        C\bke{\mu_0|Dv-Dw|^2+|Dv-Dw|^q},
        \end{equation}
      \begin{equation} \label{Coer-D-1}
       \bke{G[s]+G^\prime[s]s}s\approx(\mu_0+s)^{\frac{q}{2}}.
        \end{equation}

 Under the assumption that a nonnegative $f_0 \in
(L^\infty\cap L^1)(\mathbb{T}^3 \times \R^3)$ with compact support
and $u_0 \in W^{1,2}(\mathbb{T}^3)$, existence of weak solutions was
established for the system \eqref{flocking-f}-\eqref{flocking-u} in
the class (see \cite[Theorem 2.1]{MPP18})
\[
 f\in L^\infty\bke{(0, T); L^1\cap L^\infty (\mathbb{T}^3 \times \R^3)}, \qquad  |v|^2f\in  L^\infty\bke{[0, T];L^1(\mathbb{T}^3 \times \R^3)},
\]
\[
u \in C\bke{[0, T]; L^2(\mathbb{T}^3}) \cap L^\infty\bke{0, T;W^{1,2}(\mathbb{T}^3)}\cap L^2\bke{0, T;W^{2,2}(\mathbb{T}^3)},
\]
\[
\nabla u \in  L^\infty\bke{0, T; L^q(\mathbb{T}^3)} \cap L^q\bke{0,
T;L^{3q}(\mathbb{T}^3)},\quad u_t \in L^2\bke{0, T;
L^2(\mathbb{T}^3)}.
\]
It was also shown via a Lyapunov functional approach that
convergence of flocking states is made exponentially fast in time as
well.

The second and third authors with their collaborators also proved,
independently, global--in--time existence of weak solutions for the
system \eqref{flocking-f}-\eqref{flocking-u} with power-law types of
$G$ including the degenerate case i.e.
$G[s]=(\mu_0+s)^{\frac{q-2}{2}}$ with $\mu_0\geq 0$ and $q>2$. More
precisely, if the initial data $f_0 \in L^\infty\cap
L^1(\mathbb{T}^3 \times \R^3)$ with the compact support and $u_0 \in
L^{2}(\mathbb{T}^3)$, weak solutions exist in the class
\begin{equation} \label{G-class-1000}
 f\in L^\infty\bke{(0, T); L^1\cap L^\infty (\mathbb{T}^3 \times \R^3)}, \qquad  |v|^2f\in  L^\infty\bke{[0, T];L^1(\mathbb{T}^3 \times \R^3)},
\end{equation}
\begin{equation} \label{G-class-2000}
u \in L^\infty\bke{[0, T]; L^2(\mathbb{T}^3)} \cap L^q\bke{0,
T;W^{1,q}(\mathbb{T}^3)},
\end{equation}
\begin{equation} \label{G-class-3000}
L^2\bke{0, T;W^{1,2}(\mathbb{T}^3)}, \qquad \mbox{if}\quad \mu_0>0.
\end{equation}

As following almost the same arguments in \cite{HKKP18}, we also can
establish existence of global-in-time weak solutions to \eqref{main}
- \eqref{initial}, provided that $G$ satisfies structure conditions
\eqref{G : property}, \eqref{Coer-D} and \eqref{monotonicity} with
$q>2$. Since its verification is rather tedious repetitions of the
arguments in \cite{HKKP18}, we just give the statement without
proof.

\begin{theorem}
        Let $T>0$. Suppose that initial data $f_0 \in (L^{1}\cap L^\infty) (\bbr^3\times \bbr^3)$
        and $u_0 \in L^{2}(\bbr^3)$ satisfy
        \begin{align*}
        \begin{aligned}
        & (i)~\iint_{\bbr^3 \times \bbr^3} f_{0} \, dv \, dx = 1,   \quad \iint_{\bbr^3 \times \bbr^3}|v|^2f_0 \, dv \, dx <\infty,   \quad f_0 \geq 0. \cr
        & (ii)~\mbox{supp}_{v} f_0 ~\mbox{is bounded in $\bbr^3$
            for each $x \in \bbr^3$}.
        \end{aligned}
        \end{align*}
          Assume further that $G$ satisfies \eqref{G : property}, \eqref{Coer-D} and \eqref{monotonicity} with
        $q>2$.
        Then, there exists a global weak solution
        $(f,u)$ in the class \eqref{G-class-1000}-\eqref{G-class-3000}
        to the equations \eqref{main}-\eqref{initial}.
\end{theorem}

Our main objective is to prove existence of strong solutions for the
system \eqref{main}-\eqref{initial} where strong solutions are
defined below (see Definition \ref{strong}). More precisely, we
establish the well-posedness of strong solutions locally in time
(Theorem \ref{thm1}).


One of main tools is to use Schauder fixed point theorem via the
weighted estimate for $f$ (Lemma \ref{well-f-without-support}
below). In particular, we use the optimal transport technique to
show the stability of solutions to  the Vlasov equation with respect
to the velocity vector $u$ as well as the existence of solutions to
the Vlasov equation. To prove the stability of solutions, we borrow
the idea of the paper of Han-Kwan et. al. \cite{HMMM17}. However,
unlike \cite{HMMM17}, we exploit estimates on the Wasserstein
distance instead of Loeper's functional.

We introduce the notion of strong solutions to
\eqref{main}-\eqref{initial}.

\begin{definition}\label{strong}
Let $T>0$ and $k\geq3$. We say that $(f, u)$ is a strong solution to
\eqref{main}-\eqref{initial} if the following conditions are
satisfied:
\begin{itemize}
\item[(i)]$f(x,\xi,t) \geq 0$ for all $(x,\xi,t)\in \bbr^3\times
\bbr^3\times (0,T]$.
\item[(ii)] $
(1+|v|^{k})\nabla^\alpha_v\nabla_x^\beta f\in L^{2}(\bbr^3\times
\bbr^3)$, where $0\le |\alpha|+|\beta|\leq 2$.
\item[(iii)]$u \in C(0, T ;
H^3(\bbr^3))\cap L^{2}(0, T ; H^4(\bbr^3))$.
\item[(iv)] $(f,u)$ solves the equations \eqref{main} - \eqref{initial} in the
pointwise sense.
\end{itemize}
\end{definition}

Now, we are ready to state the main result.

\begin{theorem} \label{thm1}
Let $k\geq3$. Suppose that
\[
(1+|x|^2+|v|^2)f_0(x,v) \in L^1(\R^3\times \R^3),\quad u_0\in H^3(\bbr^3).
\]
Assume further that $f_0(x, v)\geq0$ satisfies for $p> \max\bket{5,\frac{2k+3}{2}}$
\begin{equation}\label{Condition-1 : f_0-main}
(1+|v|^p)\nabla_v^\alpha\nabla_x^\beta f_0 \in  L_v^2
L_x^2(\R^3\times \R^3),\qquad 0\le |\alpha|+|\beta|\leq 2.
\end{equation}
Then, there exists $T_*>0$ such that the system
\eqref{main}-\eqref{initial} has a unique strong solution $(f, u)$
on the time-interval $(0, T_*]$ in the sense of Definition
\ref{strong}.
\end{theorem}

\begin{remark}
We note that the local time $T^*$ in Theorem \ref{thm1} is dependent
on the size of data in \eqref{Condition-1 : f_0-main} and
$\|u_0\|_{H^3(\R^3)}$.
\end{remark}
\begin{remark}\label{imply_condition}
One can observe that the
assumption \eqref{Condition-1 : f_0-main} implies the following: 
\begin{itemize}
\item[$\bullet$] $f_0(x,v)  \leq  \frac{C_2}{1+|v|^p}$ due to the embedding w.r.t x-
variable (see Section 3.1.2).
\item[$\bullet$] $(1+|v|^k)\nabla_v^\alpha\nabla_x^\beta f_0 \in  L^2(\R^3\times \R^3)$ if
$k<(2p-3)/2$
by the H\"{o}lder inequality (see Proposition \ref{vlasov_pro}
below).
\end{itemize}
These will be used later in the proof of Theorem \ref{thm1}.
\end{remark}

%
%

\begin{remark}
The novelty of this paper is that, compared to known results, one of
main difficulties is of course caused by nonlinear structure of
viscous part of fluid, which we successfully controlled somehow.
Another improvement is that initial data of Vlasov equation is not
assumed to be compactly supported, which makes arguments a bit
complicated for solvability of the Vlasov equation. To the knowledge
of authors, previous results seem to suppose compactly supported
initial data, and we do not know, however, if the decay condition
\eqref{Condition-1 : f_0-main} is optimal or not.
\end{remark}

This paper is organized as follows: In Section 2, we review some
preliminary results. Section 3 is devoted to the study of Vlasov
equation. In section 4, we prove Theorem \ref{thm1} is presented. In Section 5, the convergence of two strong
solutions of the fluid equation involving the drag force is shown.

\section{Preliminary}

We first introduce some notations. Let $(X,\|\cdot\|)$ be a normed
space. By $L^q(0,T;X)$, we denote the space of all Bochner
measurable functions $\varphi:(0,T)\rightarrow X$ such that
\begin{equation*}
\left\{
\begin{array}{ll}
\|\varphi\|_{L^q(0,T;X)}:=\Big(\int_0^T \norm{\varphi(t)}^qdt\Big)^{\frac{1}{q}}<\infty, \hspace{1.5cm}\mbox{if} \quad 1\leq q<\infty,\\
\vspace{-3mm}\\
\|\varphi\|_{L^{\infty}(0,T;X)}:=\displaystyle\sup_{t\in (0,T)}\norm{\varphi(t)}<\infty, \quad\quad\ \mbox{if} \quad  q=\infty.\\
\end{array}\right.
\end{equation*}
 For $1\leq q\leq \infty$, we mean by
$W^{k,q}(\bbr^3)$ the usual Sobolev space. In particular, for $q=2$,
we write $W^{k,q}(\bbr^3)$ as $H^{k}(\bbr^3)$. Let
$A=(a_{ij})_{i,j=1}^3$ and $B=(b_{ij})_{i,j=1}^3$ be matrix valued
maps, we then denote
\[
A : B = \sum_{i,j=1}^3 a_{ij}b_{ij},\qquad \nabla A : \nabla B
=\sum_{i,j=1}^3 \nabla a_{ij} \cdot \nabla b_{ij}, \qquad \nabla^2 A
: \nabla^2 B = \sum_{i,j=1}^3 \nabla^2 a_{ij} : \nabla^2 b_{ij}.
\]
The letter $C$ is used to represent a generic constant, which may
change from line to line.

Next we recall the Aubin-Lions Lemma, which will be used for compactness (see e.g.
\cite{Te}).
\begin{lemma}\label{lion}
 Let $\mathcal{B}$ be a Banach space and
$\mathcal{B}_i$, $i=0,1$, separable, reflexive Banach spaces, and
suppose that
\[  \mathcal{B}_0\hookrightarrow\hookrightarrow
\mathcal{B}\hookrightarrow \mathcal{B}_1, \] where $\hookrightarrow$
denotes continuous imbedding and $\hookrightarrow\hookrightarrow$
compact imbedding. Let
\[
W :=\{v \in L^{p_0}(I;\mathcal{B}_0)\quad \text{with} \quad
\frac{dv}{dt}\in L^{p_1}(I;\mathcal{B}_1)\}
\]
for some finite interval $I \subset \bbr$, where $p_i$ satisfying
$1<p_i<\infty$. Then, we have
 \[ W \hookrightarrow\hookrightarrow
L^{p_0}(I;\mathcal{B}). \]
\end{lemma}

We also remind a version of Schauder's fixed point theorem.
\begin{lemma}\label{Schauder}
Let $V$ be a Banach space and $K$ be a nonempty convex
closed subset of $V$. If $T : K \rightarrow K$ is a continuous
mapping and $T(K)$ is contained in a compact subset of $K$, then $T$
has a fixed point.
\end{lemma}
\smallskip

Next, we introduce a priori estimate, which is one of key estimates
involving third derivatives of $u$ (refer \cite[Lemma 2.1]{KKK1}).
\begin{lemma}\label{deriv-G}
Let $l$ be a positive integer, $\tilde{\sigma}_l :\{1,2,\cdots,
l\}\rightarrow \{1,2,\cdots, l \}$ a permutation of $\{1,2,\cdots, l
\}$, and $\pi_l$ a mapping from $\{1,2,\cdots, l\}$ to $\{1,2,3\}$.
Suppose that $u\in C^{\infty}(\R^3)\cap H^l(\R^3)$. Assume further
that $G : [0,\infty) \rightarrow [0,\infty)$ is infinitely
differentiable and satisfies properties given in \eqref{G :
property}. Then, the multi-derivative of $G$ can be rewritten as the
following decomposition:
\begin{equation*}\label{eq-1 : derive G}
\partial_{x_{\sigma_l(l)}}\partial_{x_{\sigma_l(l-1)}}
\cdots\partial_{x_{\sigma_l(1)}}G\bkt{|Du|^2} = 2\big (G'\bkt{|Du|^2} Du:
\partial_{x_{\sigma_l(l)}}\partial_{x_{\sigma_l(l-1)}}
\cdots\partial_{x_{\sigma_l(1)}}D u \big) + E_l,
\end{equation*}
where $\sigma_l:=\pi_l\circ\tilde{\sigma}_l$ and
\begin{equation*}\label{eq-2 : derive G}
E_l=2\big(\partial_{x_{\sigma(l)}} (G^{'}[|Du|^2]Du):
\partial^{l-1} Du \big) +
\partial_{x_{\sigma(l)}} E_{l-1},\qquad \quad E_1=0,
\end{equation*}
where
$\partial^{l-1}:=\partial_{x_{\sigma(l-1)}}\cdots\partial_{x_{\sigma(1)}}$.
Furthermore, we obtain the following.

\noindent $(1)$ $E_2$ and $E_3$ satisfy
\begin{equation}\label{eq-5 : derive G}
\begin{aligned}
|E_2| &\leq CG\bkt{|Du|^2}|\nabla Du|^2, \\
|E_3| &\leq CG\bkt{|Du|^2}\big(|\nabla Du|^3 +|\nabla^2 Du||\nabla Du| \big). \\
\end{aligned}
\end{equation}

\noindent $(2)$ For $1\leq \alpha \leq l$
\[
\norm{ \partial^{\alpha}G\bkt{|Du|^2}~ \partial^{l-\alpha}
Du}_{L^2}+\norm{E_\alpha
\partial^{l-\alpha}Du }_{L^2}
\]
\begin{equation}\label{eq-3 : derive G}
\leq C \norm{G[ |Du|^2 ]}_{L^\infty} \bke{ \norm{ Du}_{L^\infty} + \norm{
Du}_{L^\infty}^\alpha }\norm{\nabla^l Du}_{L^2}.
\end{equation}
\end{lemma}

Next, we also introduce a monotonicity property of the viscous part of the stress
tensor, which is useful for the uniqueness of strong solutions for
fluid equations (see \cite[Lemma 2.2]{KKK1}).

\begin{lemma}\label{r-400}
Let $v,w \in W^{1,2}(\bbr^3)$. Under the Assumptions \ref{G-assume} on $G$, we have
\begin{equation*}\label{eq 2 : r-400}
m_0\norm{ Dv-Dw}^2_{L^2(\bbr^3)}\leq
\int_{\bbr^3}\bke{ G\bkt{|Dv|^2}Dv-G\bkt{|Dw|^2}Dw }: (Dv-Dw)\,dx,
\end{equation*}
where $m_0$ is a positive constant in \eqref{G : property}.
\end{lemma}

\subsection{Review on Maximal functions}

Let us first remind the notion of Maximal functions. For every $g\in
L^1_{loc}(\bbr^3)$, the associated maximal function, denoted $Mg$,
is defined by
\begin{equation*}
Mg(x) := \sup_{ r>0}\frac{1}{|B_r(x)|} \int_{B_r(x)} \abs{g(y)} dy,
~~\mbox{for a.e.} ~ x \in \bbr^3,
\end{equation*}
where $B_r(x)$ is the ball of radius $r$ centered at $x$.

It is well known that if $g\in L^p(\bbr^3)~(1<p\leq \infty)$ then so is $Mg$, and there is a constant $C_p>0$ such that
\begin{equation}\label{eq 2 : maximal function}
 \norm{Mg }_p \leq C_p \norm{g}_p .
\end{equation}
 We will use the following property of the maximal function (see \cite[Lemma 3]{A-F}).

\begin{lemma}

 If $g \in W^{1,p}(\bbr^3)$ for $p \geq 1$ then, for a.e. $x, y \in \bbr^3$ one has
\begin{equation}\label{eq 1 : maximal function}
|g(x) - g(y)| \lesssim |x - y|\bke{M \nabla g(x) +M \nabla g(y)}.
\end{equation}

\end{lemma}

\subsection{Preliminary results on optimal mass transportation and Wasserstein space}
In this subsection, we introduce the Wasserstein space and remind some properties of it. For more detail, readers may refer \cite{A-G-S}.

\begin{definition}
Let $\mu$ be a probability measure on $\mathbb{R}^d$ and $T: \mathbb{R}^d\mapsto \mathbb{R}^d$ a measurable map. Then, $T$ induces
a probability measure $ \nu$ on $\mathbb{R}^d$ which is defined as
$$\int_{\mathbb{R}^d} \varphi(y)d\nu(y) = \int_{\mathbb{R}^d} \varphi(T(x))d\mu(x) \qquad \forall ~ \varphi\in C(\mathbb{R}^d).$$
We denote $\nu:=T_\#\mu$ and say that $\nu$ is the push-forward of $\mu$ by $T$.
\end{definition}

\begin{definition}
Let us denote by $\mathcal{P}_2(\mathbb{R}^d)$ the set of all Borel
probability measures on $\mathbb{R}^d$ with a finite second moment. For $\mu,\nu\in\mathcal{P}_2(\mathbb{R}^d)$, we consider
\begin{equation}\label{Wasserstein dist}
W_2(\mu,\nu):=\left(\inf_{\gamma\in\Gamma(\mu,\nu)}\int_{\mathbb{R}^d\times
\mathbb{R}^d}|x-y|^2d\gamma(x,y)\right)^{\frac{1}{2}},
\end{equation}
where $\Gamma(\mu,\nu)$ denotes the set of all Borel probability
measures on $\mathbb{R}^d\times \mathbb{R}^d$ which has $\mu$ and
$\nu$ as marginals, i.e.
$$\gamma(A\times \mathbb{R}^d)= \mu(A)\qquad \mbox{and}\qquad\gamma(\mathbb{R}^d\times A)= \nu(A) $$
for every Borel set $A\subset \mathbb{R}^d.$

Equation (\ref{Wasserstein dist}) defines a distance on
$\mathcal{P}_2(\mathbb{R}^d)$ which is called the {\it Wasserstein distance}.
 Equipped with the Wasserstein distance,  $\mathcal{P}_2(\mathbb{R}^d)$ is called the {\it Wasserstein space}.
 It is known that the infimum in the right hand side of Equation
(\ref{Wasserstein dist}) always achieved. We will denote by $\Gamma_o(\mu,\nu)$ the set of all $\gamma$ which minimize the expression.
\end{definition}

\begin{definition}
Let $\sigma:[a,b]\mapsto \mathcal{P}_2(\mathbb{R}^d)$ be a curve.
We say that $\sigma$ is absolutely continuous and denote it by $\sigma \in AC_2(a,b;\mathcal{P}_2(\mathbb{R}^d))$, if there exists $m\in
L^2([a,b])$ such that
\begin{equation}\label{AC-curve}
W_2(\sigma(s),\sigma(t))\leq \int_s^t m(r)dr\qquad \forall ~ a\leq s\leq t\leq b.
\end{equation}
If $\sigma \in AC_2(a,b;\mathcal{P}_2(\mathbb{R}^d))$, then the limit
$$
|\sigma'|(t):=\lim_{s\rightarrow t}\frac{W_2(\sigma(s),\sigma(t))}{|s-t|}
$$
exists for $L^1$-a.e $t\in[a,b]$. Moreover, the function $|\sigma'|$ belongs to $L^2(a,b)$ and satisfies
\begin{equation*}
|\sigma'|(t)\leq m(t) \qquad \mbox{for} ~L^1-\mbox{a.e.}~t\in [a,b],
\end{equation*}
for any $m$ satisfying \eqref{AC-curve}.
We call $|\sigma'|$ by the metric derivative of $\sigma$.
\end{definition}

\begin{lemma}[\cite{A-G-S}, Theorem 8.3.1]\label{representation of AC curves}
If $\sigma\in AC_2(a,b;\mathcal{P}_2(\mathbb{R}^d))$, then there exists a Borel vector field $v: \bbr^d\times(a,b)\mapsto \bbr^d$ such that
$$v_t \in L^2(\sigma_t) \mbox{   for} ~~L^1-a.e  ~~ t\in [a,b]  , $$
and the continuity equation
$$\partial_t\sigma_t +\nabla\cdot(v_t\sigma_t)=0  $$
holds in the sense of distribution sense.

Conversely, if a weak* continuous curve $\sigma : [a,b] \mapsto
\mathcal{P}_2(\mathbb{R}^d)$ satisfies the the continuity equation
for some Borel vector field $v_t$ with $||v_t||_{L^2(\sigma_t)}\in
L^2(a,b)$ then $\sigma: [a,b]\mapsto \mathcal{P}_2(\mathbb{R}^d)$ is
absolutely continuous and $|\sigma_t'|\leq ||v_t||_{L^2(\sigma_t)}$
for $L^1$-a.e $t\in [a,b]$.
\end{lemma}
{\it Notation} : In Lemma \ref{representation of AC curves}, we use
notation $v_t:=v(\cdot,t)$ and $\sigma_t:=\sigma(t)$. Throughout
this paper, we keep this convention, unless any confusion is to be
expected, and a usual notation $\partial_t$ is adopted for temporal
derivative, i. e. $f_t:=f(\cdot,t)$ and $\partial_t f:=
\frac{\partial f}{\partial t}$.

\begin{lemma}\label{Lemma - speed}
For $i=1,~2$, let $\sigma^i \in AC_2(a,b;\mathcal{P}_2(\mathbb{R}^d))$ and $\sigma^i$ satisfies
\[
\partial_t \sigma^i_t + \nabla \cdot (v^i_t \sigma^i_t)=0.
\]
We set
$$Q(t):=\frac{1}{2}W_2^2(\sigma^1_t,\sigma^2_t). $$
Then, we have $Q \in W^{1,2}(a,b)$ and, for a.e $t\in(a,b)$
\[
\frac{dQ}{dt} \leq \int_{\mathbb{R}^d\times
\mathbb{R}^d} \big(v^1_t(x)-v^2_t(y)\big)\cdot \big(x-y\big) d\gamma_t(x,y), \qquad \gamma_t \in \Gamma_o(\sigma^1_t,\sigma^2_t).
\]
\end{lemma}

\begin{proof}
Refer to Lemma 2.4.1 of \cite{K-thesis}.
\end{proof}

\section{Vlasov equation}
In this section, for given $\bar{u} \in
L^\infty\bke{0,T;H^3(\bbr^{3})}\cap L^2\bke{0,T; H^4(\bbr^{3})}$, we
consider the following linearized system of the Vlasov type equation
of \eqref{main}:
\begin{equation}\label{approx-CS-NNS-10-10}
\partial_t f +v \cdot \nabla_xf+\nabla_v \cdot \bke{(\bar{u}-v)f}=0,
  \quad (x, v) \in \bbr^3 \times \bbr^{3},~~t > 0,
\end{equation}
which requires initial condition $ f(x,v,0) = f_0(x,v)$.

\subsection{Solutions of Vlasov Equation as a curve in Wasserstein space }
\subsubsection{ODE }
For given $\bar{u} \in L^\infty(0,T;H^3(\bbr^{3}))\cap L^2(0,T; H^4(\bbr^{3})) $, we define a vector field $\Xi:[0,T]\times \bbr^3\times\bbr^3\rightarrow \bbr^3\times\bbr^3$ by
$$\Xi(t,x,v): =(v,\bar{u}(t,x)-v).$$
We also define a flow map $\Phi: [0,T]\times [0,T]\times
\bbr^{3}\times \bbr^{3} \mapsto \bbr^{3}\times \bbr^{3} $
corresponding to the vector field $\Xi$ by
$$\Phi(t;s,x,v):=(X(t;s,x,v),V(t;s,x,v)), $$ where
\begin{equation*}\label{ODE-F}
\left\{
\begin{array}{cl}
 \frac{d}{dt}X(t;s,x,v)= V(t;s,x,v),\\
 \frac{d}{dt}V(t;s,x,v)= \bar{u}(t,X(t;s,x,v))-V(t;s,x,v),\\
 X(s;s,x,v)=x, ~~ V(s;s,x,v)=v.
\end{array}\right.
\end{equation*}
We note
\[
\frac{d}{dt}\|V(t;s,x,v) \|^2 = 2V'\cdot V=2 \big( \bar{u}
-V\big)\cdot V\leq 0,
\]
for $|V|\geq |\bar{u}|$. Hence, if  $|v| \geq \|\bar{u}\|_{L^\infty(0,T: L^\infty(\bbr^3))}$ then
$|V(t;s,x,v)|\leq |v|$ for all $(t,x)\in [s,T]\times\bbr^3$.
Let $M$ be a positive number such that
$$\norm{\bar{u}}_{L^\infty(0,T;H^3(\bbr^{3}))} + \norm{\bar{u}}_{L^2(0,T; H^4(\bbr^{3}))} \leq M.$$
Then we know $\norm{\bar{u}}_{L^\infty(0,T: L^\infty(\bbr^3))} \leq C_1M$ for some $C_1>0$ . Hence, we have
\begin{equation}\label{Est - V}
|V(t;s,x,v)|\leq \max\bket{C_1M, |v|} \quad \mbox{for all}~(t,x)\in [s,T]\times\bbr^3.
\end{equation}
This implies
\begin{equation}\label{Est - X}
\begin{aligned}
|X(t;s,x,v)|&\leq |x|+ (t-s) \sup_{s\leq \tau\leq t}| V(\tau;s,x,v)|\\
&\leq |x|+(t-s)\max\bket{C_1M,|v| }.
\end{aligned}
\end{equation}

\subsubsection{Flow map generating a solution of Vlasov Equation}
Let $f_0\in L^1(\bbr^3\times\bbr^3)$ be a nonnegative function and we may assume $\|f_0\|_{L^1}=1$ without loss of generality.
That is, $f_0 \in \mathcal{P}(\bbr^3\times\bbr^3)$.
We define $f:[0,T] \mapsto \mathcal{P}(\bbr^3\times\bbr^3)$ by
\begin{equation}\label{Def - f}
f(t):= \Phi(t;0,\cdot,\cdot)_{\#}f_0.
\end{equation}
Then, $f$ is the unique solution of (refer to \cite{A-G-S})
\begin{equation*}\label{Vlasov-equ}
\partial_t f +  \nabla_x \cdot(v f) + \nabla_v\cdot\bke{(\bar{u}-v)f} = 0.
\end{equation*}
We note that the change of variable formula combined with \eqref{Def - f} gives
\begin{equation}\label{eq 25 : continuity}
f(t,X(t;0,x,v),V(t;0,x,v))= e^{3t}f_0(x,v).
\end{equation}
We also see that \eqref{eq 25 : continuity} can be written as
\begin{equation}\label{linear-form-f}
f(t,x,v)= e^{3t}f_0\bke{X(0;t,x,v), V(0;t,x,v)}.
\end{equation}
Suppose $f_0$ satisfies \eqref{Condition-1 : f_0-main}, that is there exists $C_2>0$ such that
\begin{equation}\label{Condition-1 : f_0}
f_0(x,v)  \leq  \frac{C_2}{1+|v|^p} , \quad \forall~(x,v)\in \bbr^3\times \bbr^3, ~~ p>5.
\end{equation}
Then we have
\begin{equation}\label{eq 30 : continuity}
\begin{aligned}
\rho_f(t,x):&= \int_{\bbr^3} f(t,x,v)dv\\
&= \int_{\bbr^3} e^{3t}f_0(X(0;t,x,v), V(0;t,x,v))dv.
\end{aligned}
\end{equation}
First of all, we note that if $|v| \geq C_1M \geq \|\bar{u} \|_{L^\infty(0,T;L^\infty)}$, then $|V(0;t,x,v)|\geq |v|$. Hence,
\begin{equation}\label{eq 27 : continuity}
\begin{aligned}
 \int_{|v|\geq C_1M} e^{3t}f_0(X(0;t,x,v), V(0;t,x,v))dv &\leq \int_{|v|\geq C_1M}  \frac{C_2}{1+|V(0;t,x,v)|^p}e^{3t} dv\\
 &\leq \int_{|v|\geq C_1M}  \frac{C_2}{1+|v|^p}e^{3t} dv.
\end{aligned}
\end{equation}
On the other hand, we get
\begin{equation}\label{eq 28 : continuity}
\begin{aligned}
 \int_{|v|< C_1M} e^{3t}f_0(V(0;t,x,v), V(0;t,x,v))dv \leq \pi e^{3t} \|f_0\|_{L^\infty}  (C_1M)^3.
\end{aligned}
\end{equation}
Combining \eqref{eq 30 : continuity}, \eqref{eq 27 : continuity} and \eqref{eq 28 : continuity}, we have
\begin{equation*}\label{eq 29 : continuity}
\begin{aligned}
\rho_f(t,x) \leq \tilde{C} e^{3t},
\end{aligned}
\end{equation*}
where
$$\tilde{C}:=\pi \|f_0\|_{L^\infty} (C_1M)^3 + \int_{|v|\geq C_1M}\frac{1}{1+|v|^p}dv .$$
That is $\rho_f \in L^\infty(0,T;L^\infty(\bbr^3))$. Similarly, we
have $m_2f \in L^\infty(0,T;L^\infty(\bbr^3))$ and
\[
\|m_2f(t)\|_{L^\infty}\leq \pi e^{3t} \|f_0\|_{L^\infty} (C_1M)^5 + e^{3t} \int_{|v|\geq R}\frac{|v|^2}{1+|v|^p}dv,
\]
where
$$m_2f(t,x):=\int_{\bbr^3}|v|^2f(t,x,v)dv. $$
Suppose $f_0 \in \mathcal{P}_2(\bbr^3\times\bbr^3)$, that is
\begin{equation*}\label{Condition-2 : f_0}
\iint_{\bbr^3\times\bbr^3} (|x|^2+|v|^2)f_0(x,v)dxdv < \infty.
\end{equation*}
Then, we have
\begin{equation*}
\begin{aligned}
&\iint_{\bbr^3\times\bbr^3}\big(|x|^2+|v|^2\big) f(t,x,v) dx dv \\
&= \iint_{\bbr^3\times\bbr^3} \big( |X(t;0,x,v)|^2+|V(t;0,x,v)|^2\big) f_0(x,v) dx dv\\
&\leq 2(1+t^2)\iint_{\bbr^3\times\bbr^3}\big(|x|^2+|v|^2+C_1^2M^2
\big)f_0(x,v) dx dv< \infty,
\end{aligned}
\end{equation*}
where we exploit \eqref{Est - V} and \eqref{Est - X} in the first
inequality. This says that $t \mapsto f_t$ is a curve in the
Wasserstein space $\mathcal{P}_2(\bbr^3\times\bbr^3)$. Exploiting
Lemma \ref{representation of AC curves}, we can show that it is an
absolutely continuous curve in $\mathcal{P}_2(\bbr^3\times\bbr^3)$
as follows:
\begin{equation*}
\begin{aligned}
\|\Xi(t,\cdot,\cdot) \|_{L^2_{f_t}}^2&=\iint_{\bbr^3\times\bbr^3} |\Xi(t,x,v) |^2 f(t,x,v)dxdv\\
&= \iint_{\bbr^3\times\bbr^3} \big(|v |^2 +|\bar{u}(t,x)-v|^2\big)f(t,x,v)dxdv\\
&\leq  4\iint_{\bbr^3\times\bbr^3} \big(|v |^2 +M^2\big)f(t,x,v)dxdv \\
& =  4\iint_{\bbr^3\times\bbr^3} \big(|V(t;0,x,v) |^2 + M^2\big)f_0(x,v)dxdv\\
&\leq 8 \iint_{\bbr^3\times\bbr^3} \big(|v |^2 +
M^2\big)f_0(x,v)dxdv,
\end{aligned}
\end{equation*}
which implies
\begin{equation*}
\begin{aligned}
\int_0^T \|\Xi(t,\cdot,\cdot) \|_{L^2_{f_t}} dt& \leq T\Big(8
\iint_{\bbr^3\times\bbr^3} \big(|v |^2 + M^2\big)f_0(x,v)dxdv
\Big)^{\frac{1}{2}} < \infty.
\end{aligned}
\end{equation*}
For convenience, we set $f_t:= f(t)$ and thus, we have $t \mapsto f_t \in AC_2(0,T; \mathcal{P}_2(\bbr^3\times\bbr^3))$.

\subsubsection{Estimation of Wasserstein distance}
 Compared to \cite{HMMM17} in which authors used Loeper's functional to estimate the distance between two solutions to the Vlasov equation,
we estimate the Wasserstein distance between of two solutions to the
Vlasov equation.
\begin{lemma}
Let $\bar{u}_i \in L^\infty(0,T;H^3(\bbr^{3}))\cap  L^2(0,T; H^4(\bbr^{3})) $ for $i=1,2$.
Suppose that $f_i$ is a solution of the following equation associated with $\bar{u}_i$:
\begin{equation}\label{equ - f}
\begin{aligned}
&\partial_t f_i +  \nabla_x \cdot(v f_i) + \nabla_v\cdot\bke{(\bar{u}_i-v)f_i} = 0,\\
 &\qquad\qquad\qquad    f_i(0)=f_0,
  \end{aligned}
\end{equation}
where $f_0 \in \mathcal{P}_2(\bbr^3\times\bbr^3)$ is given. We set
$Q(t):= \frac{1}{2}W_2^2(f_1(t),f_2(t))$.
Then, we have
\begin{equation}\label{eq 21 : continuity}
Q(t) \leq e^{\bke{2+\|\bar{u}_2\|_{L^\infty(0,t ; H^3(\bbr^3))}^2}t} 
\norm{\rho_1}_{L^\infty(0,t; L^\infty)}\int_0^t
\norm{(\bar{u}_1-\bar{u}_2)(s) }_{L^2}^2ds.
\end{equation}
Here, $ \rho_i(t,x):=\int_{\bbr^3} f_i(t,x,v) dv$ for $i=1, 2$.

\end{lemma}
\begin{proof}
From Lemma \ref{Lemma - speed}, we have
\begin{equation*}
\frac{d}{dt}Q(t) \leq \iint_{\bbr^3\times \bbr^3} \iint_{\bbr^3\times \bbr^3}
(v_1-v_2,\bar{u}_1(x_1,t)-\bar{u}_2(x_2,t)-v_1+v_2)\cdot
\end{equation*}
\begin{equation*}
\cdot(x_1-x_2,v_1-v_2)
d \gamma_t(x_1,v_1,x_2,v_2),
\end{equation*}
where $d \gamma_t \in \Gamma_o(f_1(t),f_2(t)) $. Hence, we have
\begin{equation}\label{eq 24 : continuity}
\begin{aligned}
\frac{d}{dt}Q(t) 
&\leq \iint_{\bbr^3\times \bbr^3} \iint_{\bbr^3\times \bbr^3} |x_1-x_2 |^2 + |v_1-v_2|^2 d \gamma_t(x_1,v_1,x_2,v_2)\\
&+  2\iint_{\bbr^3\times \bbr^3} \iint_{\bbr^3\times \bbr^3}   |\bar{u}_1(x_1,t)-\bar{u}_2(x_1,t)|^2 +  |\bar{u}_2(x_1,t)-\bar{u}_2(x_2,t)|^2 d \gamma_t(x_1,v_1,x_2,v_2)\\
&\leq 2Q(t) + 2\iint_{\bbr^3\times \bbr^3}    |\bar{u}_1(x_1,t)-\bar{u}_2(x_1,t)|^2 f_1(x_1,v_1,t) dx_1 dv_1 + \norm{\nabla \bar{u}_2}_{L^\infty}^2 Q(t)\\
&\leq (2 + \norm{\nabla \bar{u}_2}_{L^\infty}^2)Q(t) + \norm{\rho_1(t)}_{L^\infty} \norm{(\bar{u}_1-\bar{u}_2)(t) }_{L^2}^2.
\end{aligned}
\end{equation}
Exploiting Gronwall's inequality to \eqref{eq 24 : continuity} with
$Q(0)=0$, we have
\[
Q(t) \leq 
\int_0^t h(s)e^{\int_s^t g(\theta)d\theta}ds,
\]
where
\[
g(s):= 2 +  \norm{\nabla \bar{u}_2(s)}_{L^\infty}^2 \quad
\mbox{and}\quad h(s):= \norm{\rho_1(s)}_{L^\infty}
\norm{(\bar{u}_1-\bar{u}_2)(s) }_{L^2}^2.
\]
This completes the proof.
\end{proof}
\subsection{Some estimates on the Vlasov equation}
In this subsection, we present the proof of solvability of the
linear equation \eqref{approx-CS-NNS-10-10} and provide some
estimates in a weighted Lebesgue spaces. In the next lemma, we
consider the case that $f_0$ is compactly supported and smooth.

\begin{lemma}\label{well-f-without-support}
Let $T>0$, $k\in \mathbb{N}$ and $f_0 \in
C^\infty_c(\R^3\times \R^3)$ and $\bar{u}\in L^\infty(0;T
;(H^3\cap
C^\infty)(\bbr^3))\cap L^2(0;T ;(H^4\cap
C^\infty)(\bbr^3))$.
Suppose that
$f$ is the solution given in \eqref{linear-form-f} to the
equations \eqref{approx-CS-NNS-10-10}. Then, we have $(1+|v|^{k})\nabla^\alpha_v\nabla_x
^\beta f\in L^\infty(0,T;L^{2}(\bbr^3\times \bbr^3))$ and
furthermore,
\[
\sup_{0<t\leq T}\sum_{0\leq |\alpha|+|\beta|\leq
2}\norm{(1+|v|^{k})\nabla^\alpha_v\nabla_x^\beta f(t)}^2_{L^2(\bbr^3\times
\bbr^3)}
\]
\begin{equation}\label{cpt-f-vlasov}
\leq \sum_{0\leq |\alpha|+|\beta|\leq
2}\norm{(1+|v|^{k})\nabla^\alpha_v\nabla_x ^\beta
f_{0}}^2_{L^2(\bbr^3\times \bbr^3)}\exp{\Big(C\int_0^T
\bke{1+\norm{\bar{u}(\tau)}_{H^{4}(\R^3)}}\,d\tau\Big)},
\end{equation}
where $C$ is a constant depending only on $k$.
\end{lemma}

\begin{proof}
Due to the assumptions of $f_0$ and $\bar{u}$, it is well-known that
$f$ has compact support and smooth with respect to $(x,v)$ variables
(see Section 3.1.1--3.1.2). Therefore, it suffices to prove the
estimate \eqref{cpt-f-vlasov}. Indeed, we observe first that
\[
-\iint_{\bbr^3\times \bbr^3}\nabla_v\cdot
[(\bar{u}-v)\nabla^{\alpha}_v\nabla^\beta_xf]
\nabla^{\alpha}_v\nabla^\beta_xf\,dvdx
\]
\[
=3\iint_{\bbr^3\times
\bbr^3}|\nabla^{\alpha}_v\nabla^\beta_xf|^2\,dvdx-\iint_{\bbr^3\times\bbr^3}
(\bar{u}-v)\cdot \nabla_v[\nabla^{\alpha}_v\nabla^\beta_xf]
\nabla^{\alpha}_v\nabla^\beta_xf\,dvdx
\]
\[
=3\iint_{\bbr^3\times
\bbr^3}|\nabla^{\alpha}_v\nabla^\beta_xf|^2\,dvdx+\frac{1}{2}\iint_{\bbr^3\times
\bbr^3}\nabla_v\cdot(\bar{u}-v)
|\nabla^{\alpha}_v\nabla^\beta_xf|^2\,dvdx
\]
\begin{equation}\label{kinetic-monet-est-0}
=\frac{3}{2}\iint_{\bbr^3\times
\bbr^3}|\nabla^{\alpha}_v\nabla^\beta_xf|^2\,dvdx.
\end{equation}
On the other hand, we compute that
\[
-\iint_{\bbr^3\times \bbr^3}|v|^{2k} \nabla_v\cdot
[(\bar{u}-v)\nabla^{\alpha}_v\nabla^\beta_xf]
\nabla^{\alpha}_v\nabla^\beta_xf\,dvdx
\]
\[
=3\iint_{\bbr^3\times\bbr^3}|v|^{2k}
|\nabla^{\alpha}_v\nabla^\beta_xf|^2\,dvdx-\iint_{\bbr^3\times\bbr^3}|v|^{2k}
 (\bar{u}-v)\nabla_v[\nabla^{\alpha}_v\nabla^\beta_xf]
\nabla^{\alpha}_v\nabla^\beta_xf\,dvdx
\]
\[
=3\iint_{\bbr^3\times\bbr^3}|v|^{2k}
|\nabla^{\alpha}_v\nabla^\beta_xf|^2\,dvdx+\frac{1}{2}\iint_{\bbr^3\times\bbr^3}\nabla_v\cdot\Big(|v|^{2k}(\bar{u}-v)\Big)
|\nabla^{\alpha}_v\nabla^\beta_xf|^2\,dvdx
\]
\[
=\frac{1}{2}\iint_{\bbr^3\times\bbr^3}\nabla_v(|v|^{2k}) \cdot
(\bar{u}-v)|\nabla^{\alpha}_v\nabla^\beta_xf|^2\,dvdx+\frac{3}{2}\iint_{\bbr^3\times\bbr^3}|v|^{2k}|\nabla^{\alpha}_v\nabla^\beta_xf|^2\,dvdx
\]
\[
=k\iint_{\bbr^3\times\bbr^3}|v|^{2k-2}v \cdot
(\bar{u}-v)|\nabla^{\alpha}_v\nabla^\beta_xf|^2\,dvdx+\frac{3}{2}\iint_{\bbr^3\times\bbr^3}|v|^{2k}|\nabla^{\alpha}_v\nabla^\beta_xf|^2\,dvdx
\]
\[
\leq
-(k-\frac{3}{2})\iint_{\bbr^3\times\bbr^3}|v|^{2k}|\nabla^{\alpha}_v\nabla^\beta_xf|^2\,dvdx
+k\norm{\bar{u}}_{L^\infty}\iint_{\bbr^3\times\bbr^3}|v|^{2k-1}|\nabla^{\alpha}_v\nabla^\beta_xf|^2\,dvdx
\]
\[
\leq
-(k-\frac{3}{2})\iint_{\bbr^3\times\bbr^3}|v|^{2k}|\nabla^{\alpha}_v\nabla^\beta_xf|^2\,dvdx
+k\|\bar{u}\|_{L^\infty}\iint_{\bbr^3\times\bbr^3}(\frac{2k-1}{2k}|v|^{2k}+\frac{1}{2k})|\nabla^{\alpha}_v\nabla^\beta_xf|^2\,dvdx
\]
\[
\leq
\Big(-(k-\frac{3}{2})+\frac{2k-1}{2}\|\bar{u}\|_{L^\infty}\Big)\iint_{\bbr^3\times\bbr^3}|v|^{2k}|\nabla^{\alpha}_v\nabla^\beta_xf|^2\,dvdx
+\frac{1}{2}\|\bar{u}\|_{L^\infty}\iint_{\bbr^3\times\bbr^3}|\nabla^{\alpha}_v\nabla^\beta_xf|^2\,dvdx
\]
\begin{equation*}
\leq
\Big(-k+\frac{3}{2}+k\|\bar{u}\|_{L^\infty}\Big)\iint_{\bbr^3\times\bbr^3}(1+|v|^{2k})|\nabla^{\alpha}_v\nabla^\beta_xf|^2\,dvdx
\end{equation*}
\begin{equation}\label{kinetic-monet-est-00}
\leq
C(1+\|\bar{u}\|_{L^\infty})\iint_{\bbr^3\times\bbr^3}(1+|v|^{2k})|\nabla^{\alpha}_v\nabla^\beta_xf|^2\,dvdx,
\end{equation}
where $C$ is a constant depending on $k$.

From the estimate \eqref{kinetic-monet-est-0} and
\eqref{kinetic-monet-est-00}, it follows that
\[
-\iint_{\bbr^3\times \bbr^3}(1+|v|^{2k}) \nabla_v\cdot
[(\bar{u}-v)\nabla^{\alpha}_v\nabla^\beta_xf]
\nabla^{\alpha}_v\nabla^\beta_xf\,dvdx
\]
\begin{equation}\label{kinetic-monet-00}
\leq
C(1+\|\bar{u}\|_{L^\infty})\iint_{\bbr^3\times\bbr^3}(1+|v|^{2k})|\nabla^{\alpha}_v\nabla^\beta_xf|^2\,dvdx.
\end{equation}

Next, testing \eqref{approx-CS-NNS-10-10} by  $(1+|v|^{2k})f$ and
integrating over $\bbr^3\times \bbr^3$ with the
estimate \eqref{kinetic-monet-00}, we obtain
\[
\frac{1}{2}\frac{d}{dt}\iint_{\bbr^3\times\bbr^3}(1+|v|^{2k})|f|^2\,dvdx
= -\iint_{\bbr^3\times\bbr^3}(1+|v|^{2k}) \nabla_v\cdot
\bke{(\bar{u}-v)f} f\,dvdx.
\]
\begin{equation}\label{vlasov-lemma-100}
\leq
C(1+\|\bar{u}\|_{L^\infty})\iint_{\bbr^3\times\bbr^3}(|v|^{2k}+1)|f|^2\,dvdx.
\end{equation}

Taking the differential operator $\nabla_x$ to
\eqref{approx-CS-NNS-10-10}, the equation can be rewritten as
\begin{equation}\label{rewitten-eq-20}
\nabla_x f_t + v\cdot \nabla_x \nabla_xf + \nabla_v\cdot
\bke{(\bar{u}-v)\nabla_xf}+\nabla_v\cdot
\bke{\nabla_x(\bar{u}-v)f}=0.
\end{equation}

Testing \eqref{rewitten-eq-20} by $(1+|v|^{2k})\partial_{x}f$, we obtain
\[
\frac{1}{2}\frac{d}{dt}\iint_{\bbr^3\times\bbr^3}(1+|v|^{2k})|\nabla_x
f|^2\,dvdx =-\iint_{\bbr^3\times\bbr^3}(1+|v|^{2k}) \nabla_v\cdot
\bke{(\bar{u}-v)\nabla_xf} \nabla_xf\,dvdx
\]
\[
-\iint_{\bbr^3\times\bbr^3}(1+|v|^{2k})\nabla_v\cdot
\bke{\nabla_x(\bar{u}-v)f} \nabla_xf\,dvdx.:=I_1+I_2.
\]
Exploiting \eqref{kinetic-monet-00}, we estimate $I_1$ as follows:
\[
I_1 \leq
C(1+\|\bar{u}\|_{L^\infty})\iint_{\bbr^3\times\bbr^3}(|v|^{2k}+1)|\nabla_xf|^2\,dvdx.
\]
Using H\"{o}lder and Young's inequalities, we have
\[
I_{2}
\leq C\|\nabla
\bar{u}\|_{L^\infty}\iint_{\bbr^3\times\bbr^3}(1+|v|^{2k})\Big(|\nabla_vf|^2+
|\nabla_xf|^2\Big)\,dvdx.
\]
Hence, combining $I_{1}$ with $I_{2}$, we obtain
\[
\frac{d}{dt}\iint_{\bbr^3\times\bbr^3}(1+|v|^{2k})|\nabla_x
f|^2\,dvdx
\]
\[
\leq C(1+\|\bar{u}\|_{L^\infty}+\|\nabla
\bar{u}\|_{L^\infty})\iint_{\bbr^3\times\bbr^3}(|v|^{2k}+1)\Big(|\nabla_vf|^2+
|\nabla_xf|^2\Big)\,dvdx.
\]
\begin{equation}\label{vlasov-lemma-200}
\leq
C(1+\|\bar{u}\|_{H^4})\iint_{\bbr^3\times\bbr^3}(|v|^{2k}+1)\Big(|\nabla_vf|^2+
|\nabla_xf|^2\Big)\,dvdx,
\end{equation}
where we use Sobolev embedding in last inequality. Again, taking the
differential operator $\nabla^2_x$ to \eqref{rewitten-eq-20}, it is
rewritten as
\begin{equation}\label{rewitten-eq-30}
\nabla_x^2 f_t + v\cdot \nabla_x \nabla_x^2f + \nabla_v\cdot
\bke{(\bar{u}-v)\nabla_x^2f}+2\nabla_v\cdot
\bke{\nabla_x(\bar{u}-v)\nabla_xf}+\nabla_v\cdot
\bke{\nabla_x^2(\bar{u}-v)f}=0.
\end{equation}

Multiplying \eqref{rewitten-eq-30} with $(1+|v|^{2k})\nabla^2_{x}f$
and integrating over the phase variables, we obtain
\[
\frac{d}{dt}\iint_{\bbr^3\times\bbr^3}(1+|v|^{2k})|\nabla_x^2
f|^2\,dvdx =-\iint_{\bbr^3\times\bbr^3}(1+|v|^{2k})\nabla_v\cdot
[(\bar{u}-v)\nabla_x^2f] \nabla_x^2f\,dvdx
\]
\[
-2\iint_{\bbr^3\times\bbr^3}(1+|v|^{2k}) \nabla_v\cdot
[\nabla_x(\bar{u}-v)\nabla_xf]\nabla_x^2f\,dvdx
\]
\[
-\iint_{\bbr^3\times\bbr^3}(1+|v|^{2k})\nabla_v\cdot
[\nabla_x^2(\bar{u}-v)f]\nabla_x^2f\,dvdx:=II_1+2II_2+II_3.
\]
Using \eqref{kinetic-monet-00}, H\"{o}lder and Young's inequalities,
the terms $II_1$, $II_2$ and $II_3$ are estimated, respectively, as
follows.
\[
II_1 \leq
C(1+\|\bar{u}\|_{L^\infty})\iint_{\bbr^3\times\bbr^3}(1+|v|^{2k})|\nabla^2_xf|^2\,dvdx,
\]
\[
II_2\leq \|\nabla
\bar{u}\|_{L^\infty}\iint_{\bbr^3\times\bbr^3}(1+|v|^{2k})\Big(|\nabla_v
\nabla_xf|^2+|\nabla_x^2f|^2\Big)\,dvdx,
\]
and
\[
II_3\leq \|\nabla^2
\bar{u}\|_{L^\infty}\iint_{\bbr^3\times\bbr^3}(1+|v|^{2k})\Big(|\nabla_vf|^2+
|\nabla_x^2f|^2\Big)\,dvdx.
\]
Combining estimates $II_{1}$--$II_{3}$, we get
\[
\frac{1}{2}\frac{d}{dt}\iint_{\bbr^3\times\bbr^3}(1+|v|^{2k})|\nabla_x^2
f|^2\,dvdx
\]
\begin{equation}\label{vlasov-lemma-300}
\leq
C(1+\|\bar{u}\||_{H^4})\iint_{\bbr^3\times\bbr^3}(1+|v|^{2k})\Big(|\nabla_vf|^2+|\nabla_v\nabla_xf|^2+
|\nabla_x^2f|^2\Big)\,dvdx.
\end{equation}
Taking $\nabla_v$ to \eqref{approx-CS-NNS-10-10}, we have
\begin{equation}\label{rewitten-eq-dff-v-1}
\nabla_v f_t +v\cdot \nabla_x \nabla_vf+ \nabla_v\cdot
\bke{(\bar{u}-v)\nabla_vf} +\nabla_v\cdot
\bke{\partial_v(\bar{u}-v)f}=-\nabla_xf.
\end{equation}
Again, testing $(1+|v|^{2k})\nabla_{v}f$ and integrating over the
phase variables, we obtain
\[
\frac{1}{2}\frac{d}{dt}\iint_{\bbr^3\times\bbr^3}(1+|v|^{2k})|\nabla_v
f|^2\,dvdx =-\iint_{\bbr^3\times\bbr^3}(1+|v|^{2k}) \nabla_v\cdot
[(\bar{u}-v)\nabla_vf] \nabla_vf\,dvdx
\]
\[
-\iint_{\bbr^3\times\bbr^3}(1+|v|^{2k})\nabla_v\cdot
[\nabla_v(\bar{u}-v)f]
\nabla_vf\,dvdx-\iint_{\bbr^3\times\bbr^3}(1+|v|^{2k})\nabla_xf
\nabla_v f\,dvdx
\]
\[
\leq
C(1+\|\bar{u}\|_{L^\infty})\iint_{\bbr^3\times\bbr^3}(|v|^{2k}+1)
(|\nabla_xf|^2+|\nabla_vf|^2)\,dvdx,
\]
where Young's inequality is used for last term. Therefore,
\[
\frac{1}{2}\frac{d}{dt}\iint_{\bbr^3\times\bbr^3}(1+|v|^{2k})|\nabla_v
f|^2\,dvdx
\]
\begin{equation}\label{vlasov-lemma-400}
\leq
C(1+\|\bar{u}\|_{H^4})\iint_{\bbr^3\times\bbr^3}(1+|v|^{2k})(|\nabla_xf|^2+|\nabla_vf|^2)\,dvdx.
\end{equation}
Taking $\nabla_v$ to \eqref{rewitten-eq-20}, we get
\begin{equation}\label{rewitten-eq-50}
\nabla_v\nabla_x f_t + v\cdot \nabla_x \nabla_v\nabla_xf +
\nabla_v\cdot \bke{(\bar{u}-v)\nabla_v\partial_xf}+\nabla_v\cdot
\bke{\nabla_v\nabla_x(\bar{u}-v)f}
\end{equation}
\[
=-\nabla_x \nabla_xf+\nabla_v \nabla_xf-\nabla_v\cdot
\bke{\nabla_x\bar{u}\nabla_vf}.
\]
Testing \eqref{rewitten-eq-50} by $(1+|v|^{2k})\nabla_v\nabla_x f$
and using integration by parts, we obtain
\[
\frac{1}{2}\frac{d}{dt}\iint_{\bbr^3\times\bbr^3}(1+|v|^{2k})|\nabla_v\nabla_x
f|^2\,dvdx
\]
\begin{equation}\label{vlasov-lemma-500}
\leq C(1+\|\nabla
\bar{u}\|_{H^4})\iint_{\bbr^3\times\bbr^3}(1+|v|^{2k})\Big(|\nabla_v\nabla_x
f|^2+|\nabla_v\nabla_x f|^2+|\nabla^2_v f|^2+|\nabla^2_x
f|^2\Big)\,dvdx.
\end{equation}
Lastly, taking $\nabla_v$ to \eqref{rewitten-eq-dff-v-1}, we have
\begin{equation}\label{rewitten-eq-60}
\nabla^2_v f_t + v\cdot \nabla_x \partial^2_vf + \nabla_v\cdot
\bke{(\bar{u}-v)\nabla^2_vf}=2\nabla^2_vf-2\nabla_x\nabla_v f.
\end{equation}
Testing \eqref{rewitten-eq-60} by $(1+|v|^{2k})\partial^2_v f$ and
integrating over the phase variables, by the direct calculations, we
obtain
\[
\frac{1}{2}\frac{d}{dt}\iint_{\bbr^3\times\bbr^3}(1+|v|^{2k})|\nabla_v\nabla_v
f|^2\,dvdx
\]
\begin{equation}\label{vlasov-lemma-600}
\leq
C(1+\|\bar{u}\|_{L^\infty})\iint_{\bbr^3\times\bbr^3}(1+|v|^{2k})(|\nabla^2_vf|^2+|\nabla_x\nabla_vf|^2)\,dvdx.
\end{equation}
Summing up the estimates \eqref{vlasov-lemma-100},
\eqref{vlasov-lemma-200}, \eqref{vlasov-lemma-300},
\eqref{vlasov-lemma-400}, \eqref{vlasov-lemma-500} and
\eqref{vlasov-lemma-600}, we obtain
\[
\frac{1}{2}\frac{d}{dt}\iint_{\bbr^3\times\bbr^3}(1+|v|^{2k})\Big(|f|^2+|\nabla_xf|^2+|\nabla_vf|^2+|\nabla_v\nabla_xf|^2+|\nabla^2_xf|^2+|\nabla^2_vf|^2
\Big)
\]
\[
\leq
C(1+\|\bar{u}\|_{H^4})\iint_{\bbr^3\times\bbr^3}(1+|v|^{2k})\Big(|f|^2+|\nabla_xf|^2+|\nabla_vf|^2+|\nabla_v\nabla_xf|^2+|\nabla^2_xf|^2+|\nabla^2_vf|^2
\Big)\,dvdx.
\]
Put
$X(t):=\iint_{\bbr^3\times\bbr^3}(1+|v|^{2k})\Big(|f|^2+|\nabla_xf|^2+|\nabla_vf|^2+|\nabla_v\nabla_xf|^2+|\nabla^2_xf|^2+|\nabla^2_vf|^2\Big)\,dvdx.$
Thus, we have
\begin{equation}\label{monent-final}
\frac{1}{2}\frac{d}{dt}X(t)\leq C(1+\|\bar{u}\|_{H^4})X(t).
\end{equation}
Applying Grownwall's inequality to \eqref{monent-final} with
$\bar{u}\in L^\infty(0;T ;H^3(\bbr^3))\cap L^2(0;T ;H^4(\bbr^3))$,
we finally get the desired result, that is,
\[ \sup_{0<t\leq
T}\sum_{0\leq |\alpha|+|\beta|\leq
2}\norm{(1+|v|^{k})\nabla^\alpha_v\nabla_x^\beta f}^2_{L^2(\bbr^3\times
\bbr^3)}
\]
\[
\leq \sum_{0\leq |\alpha|+|\beta|\leq
2}\norm{(1+|v|^{k})\nabla^\alpha_v\nabla_x ^\beta
f_0}^2_{L^2(\bbr^3\times \bbr^3)}\exp{\Big(C\int_0^T
\Big(1+\norm{\bar{u}(\tau)}_{H^{4}}\Big)\,d\tau\Big)}.
\]
We complete the proof.
\end{proof}

Using the Lemma \ref{well-f-without-support}, we obtain similar results
for more general data $f_0$ and $\bar{u}$. In particular, compact support of $f_0$ is not necessary.

\begin{proposition}\label{vlasov_pro}
Let $T>0$ and $k\geq3$. For a given $\bar{u}\in L^\infty(0;T
;H^3(\bbr^3))\cap L^2(0;T ;H^4(\bbr^3))$ and the initial data $f_0$
satisfying
\begin{equation}\label{prop-1000}
\sum_{0\leq |\alpha|+|\beta|\leq
2}(1+|v|^{k})\nabla^\alpha_v\nabla_x^\beta f_0\in L^{2}(\bbr^3\times
\bbr^3),
\end{equation}
there exists a unique solution $f$ to
\eqref{approx-CS-NNS-10-10} such that
$(1+|v|^{k})\nabla^\alpha_v\nabla_x ^\beta f\in
L^\infty(0,T;L^{2}(\bbr^3\times \bbr^3))$. Furthermore, we obtain
the following estimate:
\[
\sup_{0<t\leq T}\sum_{0\leq |\alpha|+|\beta|\leq
2}\norm{(1+|v|^{k})\nabla^\alpha_v\nabla_x^\beta f}^2_{L^2(\bbr^3\times
\bbr^3)}
\]
\[
\leq \sum_{0\leq |\alpha|+|\beta|\leq
2}\norm{(1+|v|^{k})\nabla^\alpha_v\nabla_x ^\beta
f_0}^2_{L^2(\bbr^3\times \bbr^3)}\exp{\Big(C\int_0^T
\bke{1+\norm{\bar{u}(\tau)}_{H^{4}(\R^3)}}\,d\tau\Big)},
\]
where $C$ is a constant depending only on $k$.
\end{proposition}

\begin{proof}
For a proof, we introduce a cut-off function $\phi_{\delta}$ and
mollifier function $\eta_{\varepsilon}$ as follows. Let $\phi \in
C_c^\infty(\R^3\times \R^3)$ be a function such that
\begin{equation}\label{NNS-F}
\phi(x,v)=\left\{
\begin{array}{cl}
\displaystyle 1,\quad |x|^2+|v|^2\leq1,\\
\vspace{-3mm}\\
\displaystyle  0,\quad |x|^2+|v|^2\geq2,
\end{array}\right.
\end{equation}
and for $\delta>0$, we set $\phi_\delta(x,v ):= \phi(\delta x,
\delta v)$ for all $(x,v) \in \R^3\times \R^3$. For $\varepsilon>0$,
we define a mollifier $\theta_{\varepsilon}$ to satisfy
    \[
    \theta \in C_c^{\infty}, \quad \theta \geq 0, \quad
    \int_{\bbr^3}\theta=1 \quad \mbox{and} \quad  \theta_{\varepsilon}(x)=\varepsilon^3\theta \Big(\frac{x}{\varepsilon}
    \Big).
    \]
Thus we let  $
\eta^\varepsilon(x,v)=\theta^{\varepsilon}(x)\theta^{\varepsilon}(v)$,
    and let
    \[
   \bar{u}^\varepsilon=\bar{u}*\theta^{\varepsilon}(x)\quad \mbox{and} \quad
    f^\varepsilon_{\delta0}=(\phi_{\delta}f_{0})*\eta^{\varepsilon}(x,v).
    \]
Then, we define $(f_\delta^\varepsilon,\bar{u}^\varepsilon)$ as the
unique solution on $[0, T)$ to the approximated equation of the
linear equation \eqref{approx-CS-NNS-10-10}:
\begin{equation}\label{vlasov-ite}
\partial_t f_\delta^\varepsilon +  \nabla_x \cdot(v f_\delta^\varepsilon)
 + \nabla_v\cdot\big[(\bar{u}^\varepsilon-v)f_\delta^\varepsilon\big] = 0,\quad f_\delta^\varepsilon\Big|_{t=0} =
f^\varepsilon_{\delta0},
\end{equation} obtained by the method of characteristics (see
Section 3.1.1--3.1.2). We set
\[
\mathcal{X}:=\bket{g\in
L^{2}(\bbr^3\times \bbr^3) |\
(1+|v|^{k})\nabla^\alpha_v\nabla_x^\beta g\in L^{2}(\bbr^3\times
\bbr^3),\quad 0\leq |\alpha|+|\beta| \leq 2}.
\]
We then claim that
$f^\varepsilon_{\delta0}$ strongly converges to $f_0$ in
   $\mathcal{X}$,
that is,
\begin{equation}\label{ini-converge100}
\norm{(1+|v|^{k})\nabla^\alpha_v\nabla_x^\beta
(f^\varepsilon_{\delta0}-f_{0})}_{L^2(\bbr^3\times \bbr^3)}
\rightarrow 0,\quad \mbox{as}\ \delta\rightarrow 0 \ \mbox{and}\
\varepsilon\rightarrow 0.
\end{equation}
The proof of \eqref{ini-converge100} is as follows: As mentioned, due to the decay
assumption \eqref{Condition-1 : f_0-main} in Theorem \ref{thm1}, we
prove that $(1+|v|^k)\nabla_v^\alpha\nabla_x^\beta f_0 \in
L^2(\R^3\times \R^3)$. Indeed,
\[
\iint_{\R^3\times\R^3}|(1+|v|^k)\nabla_v^\alpha\nabla_x^\beta
f_{0}|^2\,dv\,dx
=\iint_{\R^3\times\R^3}|(1+|v|^p)\nabla_v^\alpha\nabla_x^\beta
f_{0}|^2\Big|\frac{1+|v|^k}{1+|v|^p}\Big|^2\,dv\,dx
\]
\[
\leq \int_{\R^3}
\Big|\frac{1+|v|^k}{1+|v|^p}\Big|^2\,dv\int_{\R^3}\norm{(1+|v|^p)\nabla_v^\alpha\nabla_x^\beta
f_{0}}^2_{L^\infty(\R^3)}\,dx
\]
\begin{equation}\label{aaaaa}
\leq C\int_{\R^3}\|(1+|v|^p)\nabla_v^\alpha\nabla_x^\beta
f_{0}\|^2_{L^\infty(\R^3)}\,dx<\infty,
\end{equation}
where we use $\int_{\R^3}
\Big|\frac{1+|v|^k}{1+|v|^p}\Big|^2\,dv<\infty$ (since $p>
\frac{2k+3}{2}$).
Considering the estimate \eqref{aaaaa}, due to the definition of $
f^\varepsilon_{\delta0}$, for any $\varsigma$ we can choose $R_0>0$
such that for a sufficiently small $\varepsilon$ and $\delta$
\begin{equation}\label{aaaaa-100}
\int_{\R^3}\int_{|v|> R_0}|(1+|v|^k)\nabla_v^\alpha\nabla_x^\beta
(f^\varepsilon_{\delta0}-f_{0})|^2\,dv\,dx<\frac{\varsigma}{2}.
\end{equation}
 On the other hand, we note
that
\begin{equation}\label{aaaaa-200}
\int_{\R^3}\int_{|v|\leq R_0}|(1+|v|^k)\nabla_v^\alpha\nabla_x^\beta
(f^\varepsilon_{\delta0}-f_{0})|^2\,dv\,dx<\frac{\varsigma}{2}.
\end{equation}
Via the relation \eqref{aaaaa-100} and \eqref{aaaaa-200}, choosing
sufficiently small $\varepsilon$ and $\delta$, we have for any
$\varsigma$
\[
\int_{\R^3\times \R^3}|(1+|v|^k)\nabla_v^\alpha\nabla_x^\beta
(f^\varepsilon_{\delta0}-f_{0})|^2\,dv\,dx<\varsigma,
\]
which implies \eqref{ini-converge100}. Also, due to Lemma
\ref{well-f-without-support}, the approximated solution
$(f_\delta^\varepsilon,\bar{u}^\varepsilon)$ for \eqref{vlasov-ite}
satisfies the following estimate
\[
\sup_{0<t\leq T}\sum_{0\leq |\alpha|+|\beta|\leq
2}\norm{(1+|v|^{k})\nabla^\alpha_v\nabla_x^\beta
f_\delta^\varepsilon}^2_{L^2(\bbr^3\times \bbr^3)}
\]
\[
\leq \sum_{0\leq |\alpha|+|\beta|\leq
2}\norm{(1+|v|^{k})\nabla^\alpha_v\nabla_x ^\beta
f_{\delta0}^{\varepsilon}}^2_{L^2(\bbr^3\times
\bbr^3)}\exp{\Big(C\int_0^T
\Big(1+\norm{\bar{u}(\tau)}_{W^{4,2}}\Big)\,d\tau\Big)}.
\]
When $\delta$ and $\varepsilon$ go to $0$, there exist
a weak limit $\bar{f}$ such that $f^\varepsilon_{\delta}\rightharpoonup \bar{f}$(at least up to a sequence) in the weighted space $\mathcal{X}$ and
$\bar{f}$ is a solution to the equation \eqref{approx-CS-NNS-10-10}
in a the sense of distribution and moreover the solution is unique
by the standard argument. Hence, we briefly give a proof for a
unique solvability of the solution $f$ to the linear equation
\eqref{approx-CS-NNS-10-10} with the initial data \eqref{prop-1000}
for a given $\bar{u}$.
\end{proof}

\section{Proof of Theorems}

In this section, we prove the existence of a local solution of the
second equation of \eqref{main}.
For given vector field $U:\R^3\rightarrow \R^3$ with $\mbox{div}\
U=0$ and tensor field $F:\R^3\rightarrow \R^3$, we first consider
the following non-Newtonian Stokes type equations with drift
 term
\begin{equation}\label{NNS-F}
\left\{
\begin{array}{cl}
\displaystyle \partial_t u -\nabla \cdot\big( G\bkt{|Du|^2} Du \big) +
(U\cdot\nabla) u +\nabla
p = F,\\
\vspace{-3mm}\\
\displaystyle  {\rm div}\, u=0,
\end{array}\right.
\quad \mbox{ in } \,\,Q_T:=\bbr^3\times (0,\, T)
\end{equation}
with the initial condition
\begin{equation}\label{NNS-F-ini}
u(x,0)=u_0(x), \qquad x\in\R^3.
\end{equation}

We show the local well-posedness of \eqref{NNS-F}-\eqref{NNS-F-ini}
in the next lemma.
\begin{lemma}\label{local-ex-u}
Suppose that $U\in L^{\infty}(0,T;H^{3}(\bbr^3))\cap
L^{2}(0,T;H^{4}(\bbr^3))$ and
 $F\in L^{2}(0,T;H^{2}(\bbr^3))$. There exists
$T^*:=T^*(\|u_0\|_{H^3},\|U\|_{(L^{\infty}_tH^{3}_x\cap
L^{2}_tH^{4}_x) (Q_T)}, \|F\|_{L^{2}_tH^{2}_x(Q_T)})>0$ such that
there exists unique solutions $u\in
L^{\infty}(0,T^*;H^{3}(\bbr^3))\cap L^{2}(0,T^*;H^{4}(\bbr^3))$ to
the equation \eqref{NNS-F}-\eqref{NNS-F-ini}. Moreover, $u$
satisfies
\begin{equation}\label{estimate-force}
\begin{aligned}
&\frac{d}{dt}||u||^2_{H^3(\bbr^3)} +\frac{m_0}{2}\int_{\bbr^3} (|\nabla^3 Du|^2+|\nabla^2 Du|^2 + |\nabla Du|^2 + | Du|^2)\,dx\\
& \leq C\|\nabla U\|_{L^\infty}\| u\|^2_{H^{3}} +C\|F\|^2_{H^{2}}
+C\|G\bkt{|Du|^2}\|_{L^\infty}(\|Du\|^2_{L^\infty}+\|Du\|_{L^\infty})\|\nabla^2
Du\|^2_{L^2}\\
&+C(\|G\bkt{|Du|^2}\|^2_{L^\infty}+\|G\bkt{|Du|^2}\|^6_{L^\infty})(\|Du\|^2_{L^\infty}+\|Du\|^4_{L^\infty})\|\nabla^2
Du\|^{6}_{L^2}.
\end{aligned}
\end{equation}
\end{lemma}

\begin{proof}
The proof is similar to that in \cite[Proposition 3.1]{KKK2}. The only difference is the control of the class of $F$. For the sake
of convenience, we give a proof in Appendix.
\end{proof}

For proof of Theorem \ref{thm1} using Schauder fixed point theorem (Lemma \ref{Schauder}), we introduce a function space $X_M$ defined as follows: 
\begin{equation*}\label{fixed pt space}
 X_M:=\{ \ u:\bbr^3 \times [0,T)\rightarrow \bbr^3:
 \|u\|_{X}<M \quad \mbox{and}\quad \|u_t\|^2_{L^2(0,T;L^2(\bbr^3))}<\infty\}.
\end{equation*}
with the norm
$\|u\|^2_{X}:=\|u\|^2_{L^{\infty}(0,T;H^{3}(\bbr^3))}+\frac{m_0}{2}\|u\|^2_{L^2(0,T;H^{4}(\bbr^3))}$.

\begin{thm1}
First, we define a map $ \Theta: X_M\rightarrow X_M $ by $\Theta(\bar{u})=(u)$:
\begin{align*}\label{approx-CSNNS-}
\begin{aligned}
& \partial_t f +  v\cdot \nabla_x  f
 + \nabla_v\cdot\big[(\bar{u}-v)f\big] = 0,
  \quad (x, v) \in \bbr^3 \times \bbr^{3},~~t > 0, \\
& \partial_t u-\nabla \cdot (G[|Du|^2]D u)+(\bar{u}\cdot
\nabla)u+\nabla_x p  =-\int_{\bbr^3} (\bar{u} - v) f dv, \cr &
\nabla_x \cdot u = 0.
\end{aligned}
\end{align*}

\noindent $\bullet$ (Step A: $\Theta$ is well-defined): We set
\[
g(x):=C[G(x^2)(x^2+x)+(G^2(x^2)+G^6(x^2))(x^2+x^4)x^4] \quad
\mbox{and}\quad g_M:=\sup_{0\leq x\leq M} g(x),
\]
\[
F:= \int_{\bbr^3}(\bar{u}-v)f\,dv.
\]
From the estimate \eqref{estimate-force} in Lemma \ref{local-ex-u},
we know
\begin{equation*}\label{final-est-H3-local-100}
\begin{aligned}
\frac{d}{dt}&||u||^2_{H^3(\bbr^3)} +\frac{m_0}{2}\int_{\bbr^3} (|\nabla^3 Du|^2+|\nabla^2 Du|^2 + |\nabla Du|^2 + | Du|^2)\,dx\\
& \leq C\|\nabla \bar{u}\|_{L^\infty}\| u\|^2_{H^{3}}
+C\|F\|^2_{H^{2}} +Cg(\|u\|_{H^3(\bbr^3)})\|u\|^{2}_{H^3(\bbr^3)}
\end{aligned}
\end{equation*}
\begin{equation}\label{final-est-H3-local-20-200}
\begin{aligned}
 \leq C(M\| u\|^2_{H^{3}}+g_M\|u\|^{2}_{H^3(\bbr^3)})
+C\|F\|^2_{H^{2}}.
\end{aligned}
\end{equation}
Once we have $F\in L^{2}(0,T;H^{2}(\bbr^3))$, we can choose $0<T_1\leq T^*$ ($T^*$ given in Lemma
\ref{local-ex-u}) such that
\[
\|u\|^2_{L^\infty(0,T_1;H^3(\bbr^3))}+\frac{m_0}{2}\|u\|^2_{L^2(0,T_1;H^4(\bbr^3))}\leq
M.
\]
Now, we check $F\in L^{2}(0,T;H^{2}(\bbr^3))$. Indeed, we note that
\[
\abs{\int_{\R^3}vf\,dv}=\abs{\int_{|v|\leq1}vf\,dv+\int_{|v|>1}vf\,dv}
\]
\[
\leq
\int_{|v|\leq1}f\,dv+\Big(\int_{|v|>1}|v|^{2k}f^2\,dv\Big)^{\frac{1}{2}}\Big(\int_{|v|>1}\frac{1}{|v|^{2(k-1)}}\,dv\Big)^{\frac{1}{2}}
\]
\[
\leq
\int_{|v|\leq1}f\,dv+C\Big(\int_{|v|>1}|v|^{2k}f^2\,dv\Big)^{\frac{1}{2}}.
\]
Using Jensen and H\"{o}lder's inequalities, we have
\[
\int_{\bbr^3}\abs{\int_{\R^3}vf\,dv}^2\,dx\leq
C\int_{\bbr^3}\int_{|v|\leq1}f^2\,dv\,dx+C\int_{\bbr^3}\Big(\int_{|v|>1}|v|^{2k}f^2\,dv\Big)\,dx
\]
\[
\leq C\iint_{\bbr^3\times\bbr^3}(1+|v|^{2k})f^2\,dv\,dx.
\]
Thus, we have
\[
\norm{\bar{u} \int_{\R^3}vf\,dv}_{L^2(\bbr^3)}\leq
\|\bar{u}\|_{L^\infty(\bbr^3)}\norm{\int_{\R^3}vf\,dv}_{L^2(\bbr^3)}
\]
\[
\leq
C\|\bar{u}\|_{L^\infty(\bbr^3)}\bke{\|(1+|v|^{k})f\|_{L^2(\bbr^3\times
\R^3)}}.
\]
Similarly,
\[
\norm{\bar{u} \int_{\R^3}f\,dv}_{L^2(\bbr^3)}\leq
C\|\bar{u}\|_{L^\infty(\bbr^3)}\bke{\|(1+|v|^{k})f\|_{L^2(\bbr^3\times
\R^3)}}.
\]
Using the same method, we can check that
\begin{equation}\label{moment-100-10}
\norm{\int_{\R^3}f\,dv}_{H^2(\bbr^3)}+
\norm{\int_{\R^3}vf\,dv}_{H^2(\bbr^3)}\leq C\sum_{ 0\leq \beta\leq
2}\norm{(1+|v|^{k})\partial^\beta_xf}_{L^2(\bbr^3\times \R^3)}.
\end{equation}
Using the estimate \eqref{moment-100-10}, we get
\[
\norm{\bar{u} \int_{\R^3}f\,dv}_{H^2(\bbr^3)}\leq
\|\bar{u}\|_{H^2(\bbr^3)}\norm{\int_{\R^3}f\,dv}_{L^2(\bbr^3)}+
\|\bar{u}\|_{L^2(\bbr^3)}\norm{\int_{\R^3}f\,dv}_{H^2(\bbr^3)}
\]
\[
\leq C\sum_{ 0\leq \beta\leq
2}\norm{\bar{u}}_{H^2(\bbr^3)}\norm{(1+v^{k})\partial^\beta_xf}_{L^2(\bbr^3\times
\R^3)}.
\]
Finally, using the estimate \eqref{cpt-f-vlasov}, we obtain
\begin{equation}\label{drag-est}
\int_0^{T_1}\norm{\int_{\bbr^3}(\bar{u}-v)f\,dv}^2_{H^{2}(\bbr^3)}\,dt
\end{equation}
\[
\leq
C\sum_{ 0\leq \beta\leq
2}\int_0^{T_1}\bke{\norm{\bar{u}}^2_{H^2(\bbr^3)}+1}\norm{(1+|v|^{k})\partial^\beta_xf}^2_{L^2(\bbr^3\times
\R^3)}\,dt
\]
\[
\leq C\sum_{ 0\leq \beta\leq 2}\sup_{0<\tau\leq
T_1}\norm{(1+|v|^{k})\partial^\beta_xf(\tau)}^2_{L^2(\bbr^3\times
\R^3)}\int_0^{T_1}\bke{\norm{\bar{u}}^2_{H^2(\bbr^3)}+1}\,dt<\infty,
\]
which implies $F\in L^2(0,T_1;H^2(\R^3))$. 
Note that
\[
\int_0^T\|u\|^2_{H^4}\,dt\leq T\sup_{0\leq t\leq
T}\|u(t)\|^2_{H^3(\bbr^3)}+\int_0^T\|\nabla u\|^2_{H^3}\,dt
\]
\begin{equation}\label{thm2-est-100000}
\leq T\sup_{0\leq t\leq T}\|u(t)\|^2_{H^3(\bbr^3)}+ C\int_0^T\|D
u\|^2_{H^3}\,dt,
\end{equation} where Korn's inequality was used.
Applying Grownwall's inequality to
\eqref{final-est-H3-local-20-200}, we get
\begin{equation}\label{local-global-est}
\sup_{0\leq t\leq T_1}\|u(t)\|^2_{H^3(\bbr^3)}+\frac{m_0}{2}\int_0^{T_1}\|u\|^2_{H^4(\R^3)}\,dt
\end{equation}
\[
\leq C\bke{1+ \frac{T_1m_0}{2}}\|u_0\|^2_{H^3(\bbr^3)}e^{CT_1}+\int_0^{T_1}\|F\|^2_{H^2}e^{C(T_1-t)}\,dt
\]
\[
\leq C\bke{1+
\frac{T_1m_0}{2}}e^{CT_1}\Big[\|u_0\|^2_{H^3(\bbr^3)}+\sum_{ 0\leq
\beta\leq
2}\int_0^{T_1}(\|\bar{u}\|^2_{H^2(\bbr^3)}+1)\|(1+|v|^{k})\partial^\beta_xf\|^2_{L^2(\bbr^3\times
\R^3)}\,dt\Big]
\]
\[
\leq C\bke{1+
\frac{T_1m_0}{2}}e^{CT_1+C\sqrt{T_1}\sqrt{M}}\Big[\|u_0\|^2_{H^3(\bbr^3)}+\sum_{
0\leq \beta\leq 2}\|(1+|v|^k)\partial^\beta_x
f_0\|^2_{L^2(\bbr^3\times
\bbr^3)}(1+M)T_1\Big],
\]
where we use the estimate \eqref{drag-est} in second inequality with
\eqref{thm2-est-100000} and the result in Proposition
\ref{vlasov_pro} and H\"{o}lder's inequality for $\bar{u}\in X_M$ in
third inequality. If we choose a sufficiently small $0<\tilde{T}\leq
T_1$ such that
\[
C(1+
\frac{\tilde{T}m_0}{2})e^{C\tilde{T}+C\sqrt{\tilde{T}}\sqrt{M}}\Big[\|u_0\|^2_{H^3(\bbr^3)}+\sum_{
0\leq \beta\leq 2}\|(1+|v|^k)\partial^\beta_x
f_0\|^2_{L^2(\bbr^3\times \bbr^3)}(1+M)\tilde{T}\Big]<M,
\]
then $\|u\|_{X_M}<M$. Lastly, we also prove that $u_t\in
L^2(0,\tilde{T};L^2(\bbr^3))$. It is enough to show
$F:=\int_{\R^3}(\bar{u}-v)f\, dv \in L^2(0,\tilde{T};L^2(\R^3))$ due
to \eqref{reg-time-300} and \eqref{convec-2}. Indeed, by the
estimate \eqref{drag-est} for the control of the drag force and $u
\in L^\infty(0,\tilde{T};H^3(\bbr^3))\cap L^2(0,T_1;H^4(\bbr^3))$,
it follows that
\[
\int_0^{\tilde{T}}\int_{\bbr^3}|\partial_t
u|^2\,dxdt+\int_{\bbr^3}\tilde{G}[|Du(\cdot, \tilde{T})|^2]\,dx
\]
\[
\leq \int_{\bbr^3}\tilde{G}[|Du_0|^2]\,dx+
C\int_0^{\tilde{T}}\int_{\bbr^3}|\bar{u}|^2|\nabla
u|^2\,dxdt+C\int_0^{\tilde{T}}\int_{\bbr^3}|F|^2\,dxdt
\]
\[
\leq \underbrace{\int_{\bbr^3}\tilde{G}[|Du_0|^2]\,dx}_{<\infty}+
C\underbrace{\sup_{0<\tau\leq
T_1}\|u(\tau)\|^2_{H^{2}}\int_0^{\tilde{T}} \| \nabla
u\|^2_{L^{2}}}_{<\infty}+C\underbrace{\int_0^{\tilde{T}}\int_{\bbr^3}|F|^2\,dxdt}_{<\infty}<\infty,
\]
which implies $u_t\in L^2(0,\tilde{T};L^2(\bbr^3))$.

\noindent $\bullet$ (Step B: $\Theta$ is continuous w.r.t
$L^2(0,\tilde{T}:L^2(\bbr^3))$--topology): From Corollary \ref{thm :
contraction} in Section 6, we know that $\Theta : X_M \mapsto X_M$
is continuous w.r.t the topology inherited from
$L^2(0,\tilde{T}:L^2(\bbr^3))$. Due to the Aubin-Lions Lemma, $X_M$
is a compact subset of $L^2(0,\tilde{T}:L^2(\bbr^3))$ and $X_M$ is
clearly convex. Hence, the Schauder's fixed point theorem gives us
the existence of solutions for the system \eqref{main} and the
contraction property \eqref{eq-1 : contraction} says the uniqueness
of solutions.
\end{thm1}


\section{Contraction of the iteration with respect to $L^2(0,T;L^2(\bbr^3))$}

\begin{lemma}
Assume $u_0 \in H^3(\bbr^3)$ and, for $i=1,~2$, $\bar{u}_i\in
L^\infty(0,T;H^3(\bbr^3))\cap L^2(0,T; H^4(\bbr^3)) $. Let $f_i \in
AC_2(0,T; \mathcal{P}_2(\bbr^3\times\bbr^3)$ be the solution of
\eqref{equ - f} with a given $f_0 \in \mathcal{P}_2$ satisfying
\eqref{Condition-1 : f_0} and set
$\rho_i(t,x):=\int_{\bbr^3}f_i(t,x,v)dv$ for $i=1,2$. Suppose that
$u_i\in L^\infty(0,T;H^3(\bbr^3))\cap L^2(0,T; H^4(\bbr^3)) $ is the
solution of
\begin{equation*}
\begin{aligned}
 \partial_t u_i-\Delta u_i+(\bar{u}_i\cdot \nabla)u_i +\nabla_x p_i  =-\int_{\bbr^3} (\bar{u}_i - v)
f_i dv, \quad  \nabla_x \cdot u_i = 0\quad \mbox{in }\,\,Q_T:=\R^3\times(0,T)
\end{aligned}
\end{equation*}
with initial condition $u_i(0)=u_0$.
Then, we have
\begin{equation}\label{diff-solu-unique}
\|u_1-u_2 \|_{L^\infty(0,T;L^2(\bbr^3))}^2\leq
e^{4C_1T}(C_1T+C_3)\|\bar{u}_1-\bar{u}_2
\|_{L^2(0,T;L^2(\bbr^3))}^2,
\end{equation}
where $$C_1:=C_1\bke{\|\rho_1\|_{L^\infty(Q_T)},  \|
\bar{u}_2 \|_{L^\infty_tH^3_x(Q_T)},\| u_1-u_2
\|_{L^\infty_tH^3_x(Q_T)}}$$ and
$$C_3:=C_3\bke{\|\rho_1\|_{L^\infty(Q_T)},\| \bar{u}_2
\|_{L^\infty_tH^3_x(Q_T)} }.$$ The estimate \eqref{diff-solu-unique}
implies that there exists $T\ll 1$ such that
\begin{equation*}
\begin{aligned}
\norm{u_1-u_2 }_{L^2(0,T;L^2(\bbr^3))}
&\leq \frac{1}{2}\norm{\bar{u}_1-\bar{u}_2 }_{L^2(0,T;L^2(\bbr^3))}.
\end{aligned}
\end{equation*}
\end{lemma}

\begin{proof}
 We consider the equation for $\tilde{u}:=u_1-u_2$ and $\tilde{p}:=p_1-p_2$,
\begin{equation*}
\begin{aligned}
&\partial_t \tilde{u}-\nabla\cdot[G(|Du_1|^2)Du_1 -G(|Du_2|^2)Du_2] +\nabla \tilde{p} + (u_1\cdot\nabla)\tilde{u}\\
= &-((\bar{u}_1-\bar{u}_2)\cdot\nabla)u_2-\int_{\bbr^3}
f_1dv(\bar{u}_1-\bar{u}_2) + \int_{\bbr^3}(f_2-f_1)dv\bar{u}_2 + \int_{\bbr^3}
v(f_1-f_2)dv,
\end{aligned}
\end{equation*}
and we obtain
\begin{equation*}\label{eq 9 : continuity}
\begin{aligned}
\frac{1}{2}\frac{d}{dt}\|\tilde{u}\|_{L^2}^2+m_0 \|\nabla \tilde{u}
\|_{L^2}^2 &= -\int_{\bbr^3}\tilde{u}\cdot((\bar{u}_1-\bar{u}_2)\cdot
\nabla)u_2 dx
-\int_{\bbr^3} \rho_1(\bar{u}_1-\bar{u}_2)\cdot\tilde{u}dx + A(t)\\
&=I_1+I_2 + A(t),
\end{aligned}
\end{equation*}
where
\begin{equation*}
\begin{aligned}
A(t)&=\iint_{\bbr^3 \times \bbr^3} \tilde{u}\cdot v(f_1-f_2)dvdx - \iint_{\bbr^3 \times \bbr^3} \tilde{u}\cdot \bar{u}_2(f_1-f_2)dvdx.
\end{aligned}
\end{equation*}
By H\"{o}lder and Young's inequalities, we note,
\begin{equation*}\label{eq 10 : continuity}
\begin{aligned}
I_1\leq C\|\nabla u_2 \|_{L^\infty}\big(\|\tilde{u} \|_{L^2}^2
+\|\bar{u}_1-\bar{u}_2 \|_{L^2}^2 \big),
\end{aligned}
\end{equation*}
and
\begin{equation*}\label{eq 11 : continuity}
\begin{aligned}
I_2\leq C\|\rho_1 \|_{L^\infty}\big(\|\tilde{u} \|_{L^2}^2
+\|\bar{u}_1-\bar{u}_2 \|_{L^2}^2 \big).
\end{aligned}
\end{equation*}
On the other hand, we have
\begin{equation*}
\begin{aligned}
A(t)&=\iint_{\bbr^3 \times \bbr^3} \tilde{u}\cdot v(f_1-f_2)dvdx - \iint_{\bbr^3 \times \bbr^3} \tilde{u}\cdot \bar{u}_2(f_1-f_2)dvdx\\
& := A_1 +A_2.
\end{aligned}
\end{equation*}
We note that
\begin{equation*}
\begin{aligned}
A_1& = \iint_{\bbr^3 \times \bbr^3}\iint_{\bbr^3 \times \bbr^3} [\tilde{u}(x_1,t)\cdot v_1 - \tilde{u}(x_2,t)\cdot v_2] d\gamma_t(x_1,v_1,x_2,v_2)\\
&=\iint_{\bbr^3 \times \bbr^3}\iint_{\bbr^3 \times \bbr^3} \tilde{u}(x_1,t)\cdot (v_1-v_2) + (\tilde{u}(x_1,t)-\tilde{u}(x_2,t)\cdot v_2 d\gamma_t(x_1,v_1,x_2,v_2)\\
&:= A_{11}+A_{12}.
\end{aligned}
\end{equation*}
Now we estimate
\begin{equation*}
\begin{aligned}
|A_{11}|&\leq 2\iint_{\bbr^3 \times \bbr^3}\iint_{\bbr^3 \times \bbr^3} |\tilde{u}(x_1,t)|^2 + |v_1-v_2|^2 d\gamma_t(x_1,v_1,x_2,v_2) \\
&\leq 2\bke{\|\rho_1\|_{L^\infty} \| \tilde{u}\|_{L^2}^2 + Q(t)}, 
\end{aligned}
\end{equation*}
and we exploit \eqref{eq 2 : maximal function} and \eqref{eq 1 : maximal function} to get
\begin{equation*}
\begin{aligned}
|A_{12}|&\leq \iint_{\bbr^3 \times \bbr^3}\iint_{\bbr^3 \times \bbr^3} (M \nabla \tilde{u}(x_1,t)+ M \nabla \tilde{u}(x_2,t))|x_1-x_2||v_2| d\gamma_t(x_1,v_1,x_2,v_2) \\
&\leq  \varepsilon \iint_{\bbr^3 \times \bbr^3}\iint_{\bbr^3 \times \bbr^3} \big( (M\nabla \tilde{u})^2(x_1,t)+ (M\nabla \tilde{u})^2(x_2,t) \big)|v_2|^2 d\gamma_t(x_1,v_1,x_2,v_2)\\
&+ \frac{1}{\varepsilon} \iint_{\bbr^3 \times \bbr^3}\iint_{\bbr^3 \times \bbr^3} |x_1-x_2|^2 d\gamma_t(x_1,v_1,x_2,v_2)\\
&\leq  \varepsilon\iint_{\bbr^3 \times \bbr^3}\iint_{\bbr^3 \times \bbr^3} (M\nabla \tilde{u})^2(x_1,t)(|v_1-v_2|^2 + |v_1|^2) d\gamma_t(x_1,v_1,x_2,v_2)\\
&+ \varepsilon\iint_{\bbr^3 \times \bbr^3} (M\nabla \tilde{u})^2(x,t) |v|^2 f_2(x,v,t)dxdv+  \frac{1}{\varepsilon} Q(t)\\
&\leq \varepsilon \bke{ \|\nabla \tilde{u} \|_{L^\infty}^2 Q(t) + (\|m_2f_1\|_{L^\infty} +  \|m_2f_2\|_{L^\infty}) \|M\nabla \tilde{u}\|_{L^2}^2 } + \frac{1}{\varepsilon} Q(t)\\
&\leq \varepsilon \bke{ \|\nabla \tilde{u} \|_{L^\infty}^2 Q(t) + (\|m_2f_1\|_{L^\infty} +  \|m_2f_2\|_{L^\infty}) \|\nabla \tilde{u} \|_{L^2}^2 }
+ \frac{1}{\varepsilon} Q(t).
\end{aligned}
\end{equation*}
Adding above estimates,
\[
|A_1|\leq 2(\|\rho_1\|_{L^\infty} \| \tilde{u}\|_{L^2}^2 + Q(t))
+ \varepsilon \big( \|\nabla \tilde{u} \|_{L^\infty}^2 Q(t)
\]
\begin{equation}\label{eq 18 : continuity}
\begin{aligned}
+ \bke{\|m_2f_1\|_{L^\infty} +  \|m_2f_2\|_{L^\infty}} \|\nabla \tilde{u} \|_{L^2}^2 \big) + \frac{1}{\varepsilon} Q(t).
\end{aligned}
\end{equation}
On the other hand,
\[
A_2= \iint_{\bbr^3 \times \bbr^3}\iint_{\bbr^3 \times \bbr^3} \tilde{u}(x_1,t)\cdot (\bar{u}_2(x_1,t)-\bar{u}_2(x_2,t))
+ (\tilde{u}(x_1,t)-\tilde{u}(x_2,t))\cdot \bar{u}_2(x_2,t) d\gamma_t(x_1,v_1,x_2,v_2)
\]
\[
\leq \|\nabla \bar{u}_2\|_{L^\infty}\iint_{\bbr^3 \times \bbr^3}\iint_{\bbr^3 \times \bbr^3} \bke{|\tilde{u}(x_1,t)|^2 + |x_1-x_2|^2}  d\gamma_t(x_1,v_1,x_2,v_2)
\]
\[
+ \| \bar{u}_2\|_{L^\infty}\iint_{\bbr^3 \times \bbr^3}\iint_{\bbr^3 \times \bbr^3}\varepsilon( |\nabla \tilde{u}(x_1,t)|^2+ |\nabla \tilde{u}(x_2,t)|^2) + \frac{1}{\varepsilon}|x_1-x_2|^2 d\gamma_t(x_1,v_1,x_2,v_2)
\]
\begin{equation}\label{eq 19 : continuity}
\leq \|\nabla \bar{u}_2\|_{L^\infty}(\|\rho_1\|_{L^\infty}\| \tilde{u}\|_{L^2}^2 + Q(t)) +  \| \bar{u}_2\|_{L^\infty}(\varepsilon\|\nabla \tilde{u} \|_{L^2}^2+ \frac{1}{\varepsilon}Q(t)).
\end{equation}
Combining \eqref{eq 18 : continuity} and \eqref{eq 19 : continuity}, we have
\begin{equation}\label{eq 20 : continuity}
A(t) \leq C_1 \big(\|\tilde{u}\|_{L^2}^2 +Q(t)\big) + \varepsilon
C_2\|\nabla \tilde{u}\|_{L^2}^2,
\end{equation}
where $C_1=C_1(\|\rho_1\|_{L^\infty(Q_T)}, \|\bar{u}_2
\|_{L^\infty_tH^3_x(Q_T)}),  \|\tilde{u} \|_{L^\infty_t
H^3_x(Q_T)})$, and $C_2$ depends on $(\|m_2f_1\|_{L^\infty(Q_T)}$,
$\|m_2f_2\|_{L^\infty(Q_T)}$,
$\|\bar{u}_2\|_{L^\infty_tH^3_x(Q_T)})$. Hence, we have
\begin{equation}\label{eq 22 : continuity}
\begin{aligned}
&\frac{1}{2}\frac{d}{dt}\|\tilde{u}\|_{L^2}^2+ (m_0-\varepsilon
\hat{C}) \|\nabla \tilde{u} \|_{L^2}^2\\
 \leq &(\|\nabla u_2 \|_{L^\infty}+ \|\rho_1 \|_{L^\infty})\big(\|\tilde{u} \|_{L^2}^2 +\|\bar{u}_1-\bar{u}_2 \|_{L^2}^2 \big)  + C_1 \big(\|\tilde{u}\|_{L^2}^2
 +Q(t)\big)\\
\leq & C_1 \big(\|\tilde{u}\|_{L^2}^2 +Q(t)\big)+
C_3\|\bar{u}_1-\bar{u}_2 \|_{L^2}^2,
\end{aligned}
\end{equation}
where $C_3$ depends on $\|\rho_1\|_{L^\infty(Q_T)}$ and $\|
\bar{u}_2 \|_{L^\infty_tH^3_x(Q_T)}$. Plugging \eqref{eq 21 :
continuity} into \eqref{eq 22 : continuity}, we get
\begin{equation*}\label{eq 13 : continuity}
\begin{aligned}
\frac{1}{2}\frac{d}{dt}\|\tilde{u}(t)\|_{L^2}^2 &\leq  C_1
\|\tilde{u}(t)\|_{L^2}^2+ C_3 \|(\bar{u}_1-\bar{u}_2)(t) \|_{L^2}^2
+C_1e^{C_1t}\int_0^t\|(\bar{u}_1-\bar{u}_2)(s) \|_{L^2}^2 ds.
\end{aligned}
\end{equation*}
Using Gronwall's Lemma with $\tilde{u}(0)\equiv 0$, we obtain
\begin{equation}\label{eq 16 : continuity}
\|\tilde{u}(t)\|_{L^2}^2 \leq e^{2C_1t}\int_0^t(g(s)+h(s))ds,
\end{equation}
where
$$g(t):=2C_3\|(\bar{u}_1-\bar{u}_2)(t) \|_{L^2}^2 \qquad  h(t):=2C_1e^{C_1t}\int_0^t \|(\bar{u}_1-\bar{u}_2)(s) \|_{L^2}^2 ds.$$
From \eqref{eq 16 : continuity}, we have
\begin{equation*}
\|\tilde{u} \|_{L^\infty(0,T;L^2(\bbr^3))}^2\leq
e^{4C_1T}(C_3+C_1T)\|\bar{u}_1-\bar{u}_2
\|_{L^2(0,T;L^2(\bbr^3))}^2,
\end{equation*}
which gives us
\begin{equation*}
\begin{aligned}
\|\tilde{u} \|_{L^2(0,T;L^2(\bbr^3))}^2&\leq T \|\tilde{u} \|_{L^\infty(0,T;L^2(\bbr^3))}^2\\
&\leq Te^{4C_1T}(C_3+C_1T)\|\bar{u}_1-\bar{u}_2 \|_{L^2(0,T;L^2(\bbr^3))}^2\\
&\leq \frac{1}{4}\|\bar{u}_1-\bar{u}_2 \|_{L^2(0,T;L^2(\bbr^3))}^2
\end{aligned}
\end{equation*}
for some small $T\ll 1$.
\end{proof}

\begin{corollary}\label{thm : contraction}
Assume $u_0 \in H^3(\bbr^3)$ and $f_0 \in
\mathcal{P}_2(\bbr^3\times\bbr^3)$ satisfying \eqref{Condition-1 :
f_0}. Let
$$X_M = \{ u | \|u\|_{L^\infty(0,T:H^3(\bbr^3)} + \|u \|_{L^2(0,T;H^4(\bbr^3))}\leq M\}$$
and  $\Theta: X_M \rightarrow X_M$ be defined by $u:=\Theta(\bar{u})$ where
\begin{equation*}
\begin{aligned}
 &\partial_t u-\Delta u+(\bar{u}\cdot \nabla)u +\nabla_x p  =-3\int_{\bbr^3} (\bar{u} - v) f dv, \quad  \nabla_x \cdot u = 0,\\
&\partial_t f +  \nabla_x \cdot(v f) + \nabla_v\cdot\big[(\bar{u}-v)f\big] = 0,\\
 &u_i(0)=u_0,~ f(0)=f_0.
\end{aligned}
\end{equation*}
Then, for small $T\ll 1$, $\Theta$ is a contraction mapping with respect to the
topology induced by the norm $L^2(0,T;L^2(\bbr^3))$. That is, we have
\begin{equation}\label{eq-1 : contraction}
\begin{aligned}
\|u_1-u_2 \|_{L^2(0,T;L^2(\bbr^3))}
&\leq \frac{1}{2}\|\bar{u}_1-\bar{u}_2 \|_{L^2(0,T;L^2(\bbr^3))}
\end{aligned}
\end{equation}
for small $T\ll 1$.
\end{corollary}

\appendix

\section{}

In this section, as mentioned, we give a proof of Lemma
\ref{local-ex-u}.

\begin{lem5} We suppose that $u$ is regular. We then compute certain a
priori estimates. First of all, we note that by the $L^2$-energy
estimate with \eqref{G : property},
\begin{equation}\label{eq-1 : H0}
\frac{d}{dt}\|u\|^2_{L^2(\bbr^3)}+m_0\|Du\|^2_{L^2(\bbr^3)}\leq
C\|u\|^2_{L^2(\bbr^3)}+C\|F\|^2_{L^2(\bbr^3)}.
\end{equation}

\noindent$\bullet$ ($\|\nabla u \|_{L^2}$-estimate)\, Taking
derivative $\partial_{x_i}$ to \eqref{NNS-F} and multiplying
$\partial_{x_i}u$,
\begin{equation*}\label{eq-1 : H1}
\frac{1}{2}\frac{d}{dt}||\partial_{x_i}u||^2_{L^2(\bbr^3)}+\int_{\bbr^3}\partial_{x_i}
(G\bkt{|Du|^2} Du): \partial_{x_i} Du\,dx
\end{equation*}
\[
=- \int_{\bbr^3}\partial_{x_i} \big((U\cdot \nabla) u \big )\cdot
\partial_{x_i} u\,dx-\int_{\bbr^3} F \cdot
\partial_{x_i}\partial_{x_i} u\,dx.
\]
Noting that
\begin{equation*}\label{eq-2 : H1}
\begin{aligned}
\partial_{x_i}(G\bkt{|Du|^2} Du): \partial_{x_i} Du&=\big[ \partial_{x_i}G\bkt{|Du|^2} Du + G\bkt{|Du|^2}\partial_{x_i} Du \big] :  \partial_{x_i} Du\\
&= 2 G^{'}[|Du|^2](D u : \partial_{x_i}Du)( Du: \partial_{x_i} Du) + G\bkt{|Du|^2}|\partial_{x_i} Du|^2\\
&= 2 G^{'}[|Du|^2]| Du: \partial_{x_i} Du|^2 +
G\bkt{|Du|^2}|\partial_{x_i} Du|^2,
\end{aligned}
\end{equation*}
we have
\begin{equation}\label{eq-4 : H1}
\frac{1}{2}\frac{d}{dt}||\partial_{x_i}u||^2_{L^2(\bbr^3)}+\int_{\bbr^3}G\bkt{|Du|^2}|\partial_{x_i}
Du|^2\,dx + \int_{\bbr^3}2G^{'}[|Du|^2]| Du: \partial_{x_i}
Du|^2\,dx
\end{equation}
\[
= -\int_{\bbr^3}\partial_{x_i} \big((U\cdot \nabla) u \big )\cdot
\partial_{x_i} u\,dx-\int_{\bbr^3} F \cdot \partial_{x_i}\partial_{x_i} u\,dx.
\]
Taking $A=Du$ and $B=\partial_{x_i} D u$ in the last
inequality, we use the following inequality:
\begin{equation}\label{app:eq-1 : G}
G[|A|^2]|B|^2 + 2 G^{\prime}[|A|^2](A:B)^2 \geq  m_0|B|^2.
\end{equation}
Indeed, if $G^{\prime}[|A|^2] \geq 0$, then
\begin{equation*}\label{app:eq-2 : G}
G[|A|^2]|B|^2 + 2 G^{\prime}[|A|^2](A:B)^2 \geq  G[|A|^2]|B|^2 \geq
m_0|B|^2.
\end{equation*}
In case that $G^{\prime}[|A|^2] <0$, we note that
\begin{equation*}\label{app:eq-3 : G}
G[|A|^2]|B|^2 + 2 G^{\prime}[|A|^2](A:B)^2 \ge \big(G[|A|^2] + 2
G^{\prime}[|A|^2]|A|^2  \big)|B|^2\geq m_0 |B|^2.
\end{equation*}
Applying the inequality \eqref{app:eq-1 : G} to \eqref{eq-4 : H1},
we obtain
\begin{equation}\label{eq-3 : H1}
\begin{aligned}
&\frac{1}{2}\frac{d}{dt}||\partial_{x_i}u||^2_{L^2(\bbr^3)}+\int_{\bbr^3}m_0
|\partial_{x_i} Du|^2\,dx \\
\leq &-\int_{\bbr^3}\partial_{x_i}
\big((U\cdot \nabla) u \big )\cdot
\partial_{x_i} u\,dx-\int_{\bbr^3} F \cdot \partial_{x_i}\partial_{x_i} u\,dx.
\end{aligned}
\end{equation}

We will treat the term in righthand side caused by convection
together later.

\noindent$\bullet$ ($\|\nabla^2 u \|_{L^2}$-estimate)\, Taking the
derivative $\partial_{x_{j}}\partial_{x_{i}}$ on \eqref{NNS-F} and
multiplying it by $\partial_{x_{j}}\partial_{x_{i}}u $,
\begin{equation}\label{eq-2 : H2}
\frac{1}{2}\frac{d}{dt}||\partial_{x_{j}}\partial_{x_{i}}u||^2_{L^2(\bbr^3)}+\int_{\bbr^3}\partial_{x_{j}}\partial_{x_{i}}\Big[G\bkt{|Du|^2}
D u\Big]: \partial_{x_{j}}\partial_{x_{i}}D u\,dx
\end{equation}
\[
= -\int_{\bbr^3}\partial_{x_{j}}\partial_{x_{i}}\big((U\cdot\nabla
u) \big)\cdot \partial_{x_{j}}\partial_{x_{i}} u\,dx-\int_{\bbr^3}
\partial_{x_i}\partial_{x_i}F : \partial_{x_i}\partial_{x_i} u\,dx.
\]
We observe that
\begin{equation}\label{eq-3 : H2}
\begin{aligned}
&\int_{\bbr^3}\partial_{x_{j}}\partial_{x_{i}}\Big[G\bkt{|Du|^2} D u\Big]:\partial_{x_{j}}\partial_{x_{i}}D u\,dx\\
&=\int_{\bbr^3}G\bkt{|Du|^2}|\partial_{x_{j}}\partial_{x_{i}}D u|^2\,
+\sum_{\sigma}\int_{\bbr^3}\partial_{x_{\sigma(i)}}G\bkt{|Du|^2} (\partial_{x_{\sigma(j)}}D u : \partial_{x_{j}}\partial_{x_{i}}D u)\,dx\\
&\hspace{2cm}+\int_{\bbr^3}\partial_{x_{j}}\partial_{x_{i}}G\bkt{|Du|^2}(D
u :
\partial_{x_{j}}\partial_{x_{i}}D u)\,dx=:I_{21}+I_{22}+I_{23},
\end{aligned}
\end{equation}
where $\sigma :\{ i,j\}\rightarrow \{i,j \}$ is a permutation of
$\{i,j \}$. We separately estimate terms $I_{22}$ and $I_{23}$ in
\eqref{eq-3 : H2}. Using H\"{o}lder, Young's and Gagliardo-Nirenberg
inequalities, we have for $I_{22}$
\[
|I_{22}|=\abs{\int_{\bbr^3}  2G^{'}[|Du|^2](Du:
\partial_{x_{\sigma(i)}}D u)(\partial_{x_{\sigma(j)}}D u :
\partial_{x_{j}}\partial_{x_{i}}Du)\,dx}
\]
\[
\leq  C\|G\bkt{|Du|^2}\|_{L^\infty}\|\nabla Du\|^2_{L^4}\|\nabla^2
Du\|_{L^2}
\]
\begin{equation*}\label{eq-5 : H2}
\leq   C\|G\bkt{|Du|^2}\|_{L^\infty}\|Du\|_{L^\infty}\|\nabla^2
Du\|^2_{L^2},
\end{equation*}
where we used the condition \eqref{G : property}.

\noindent  For $I_{23}$, using Lemma \ref{deriv-G}, we compute
\begin{equation*}
\begin{aligned}
I_{23}&= \int_{\bbr^3}2\big( G^{'}[|Du|^2](Du : \partial_{x_{j}}\partial_{x_{i}}Du) + E_2 \big) (D u : \partial_{x_{j}}\partial_{x_{i}}Du)\,dx\\
& =\int_{\bbr^3}E_2(D u : \partial_{x_{j}}\partial_{x_{i}}Du)\,dx
+ 2\int_{\bbr^3}G^{'}[|Du|^2]  |Du: \partial_{x_{j}}\partial_{x_{i}}Du|^2\,dx\\
&:=I_{231}+I_{232}.
\end{aligned}
\end{equation*}
The term $I_{231}$ is estimated as
\begin{equation}\label{eq-8 : H2}
\begin{aligned}
\abs{ I_{231}}&\leq C\|G\bkt{|Du|^2}\|_{L^\infty} \|Du\|_{L^\infty}\|\nabla Du\|^2_{L^4}\|\nabla^2 Du\|_{L^2}\\
 &\leq  C\|G\bkt{|Du|^2}\|_{L^\infty} \|Du\|_{L^\infty}^2\|\nabla^2
 Du\|_{L^2}^2,
\end{aligned}
\end{equation}%
where we used the first inequality of \eqref{eq-5 : derive G}. We
combine estimates \eqref{eq-2 : H2}-\eqref{eq-8 : H2} to get
\[
\frac{1}{2}\frac{d}{dt}||\partial_{x_{j}}\partial_{x_{i}}u||^2_{L^2(\bbr^3)}+\int_{\bbr^3}G\bkt{|Du|^2}|\partial_{x_{j}}\partial_{x_{i}}D
u|^2 + \int_{\bbr^3}2G^{'}[|Du|^2]|Du:
\partial_{x_{j}}\partial_{x_{i}}Du|^2
\]
\begin{equation*}\label{eq-4 : H2}
\leq C\|G\bkt{|Du|^2}\|_{L^\infty}(\|Du\|_{L^\infty} +
\|Du\|_{L^\infty}^2)\|\nabla^2 Du\|^2_{L^2}
-\int_{\bbr^3}\partial_{x_{j}}\partial_{x_{i}}\big((u\cdot\nabla u)
\big)\cdot \partial_{x_{j}}\partial_{x_{i}} u.
\end{equation*}
Similarly as in \eqref{eq-3 : H1}, we have
\[
\frac{1}{2}\frac{d}{dt}||\partial_{x_{j}}\partial_{x_{i}}u||^2_{L^2(\bbr^3)}+\int_{\bbr^3}m_0|\partial_{x_{j}}\partial_{x_{i}}D
u|^2
\]
\begin{equation}\label{eq-1 : H2}
\leq C\|G\bkt{|Du|^2}\|_{L^\infty}(\|Du\|_{L^\infty} +
\|Du\|_{L^\infty}^2)\|\nabla^2 Du\|^2_{L^2}
\end{equation}
\[
-\int_{\bbr^3}\partial_{x_{j}}\partial_{x_{i}}\big((U\cdot\nabla u)
\big)\cdot \partial_{x_{j}}\partial_{x_{i}} u-\int_{\bbr^3}
\partial_{x_i}F \cdot \partial_{x_j}\partial_{x_j}\partial_{x_i}
u\,dx.
\]
\noindent$\bullet$ ($\|\nabla^3 u \|_{L^2}$-estimate) For
convenience, we denote
$\partial^3:=\partial_{x_{k}}\partial_{x_{j}}\partial_{x_{i}}$.
Similarly as before, taking the derivative $\partial^3$ on
\eqref{NNS-F} and multiplying it by $\partial^3 u$,
\begin{equation}\label{eq-1 : H3}
\frac{1}{2}\frac{d}{dt}||\partial^3u||^2_{L^2(\bbr^3)}
+\int_{\bbr^3}\partial^3\Big[G\bkt{|Du|^2} D u\Big] :
\partial^3D u\,dx
\end{equation}
\[
-\int_{\bbr^3}\partial^3\big((U\cdot\nabla u) \big)\cdot\partial^3
u\,dx-  \int_{\bbr^3}\partial^3F:\partial^3 u\,dx.
\]
Direct computations show that
\[
\int_{\bbr^3}\partial^3\Big[G\bkt{|Du|^2} D u\Big] : \partial^3D
u\,dx=\int_{\bbr^3}G\bkt{|Du|^2}|\partial^3D u|^2\,dx
\]
\[
+\sum_{\sigma_3}\int_{\bbr^3}\partial_{x_{\sigma_3(i)}}G\bkt{|Du|^2}(\partial_{x_{\sigma_3(k)}}\partial_{x_{\sigma_3(j)}}D
u:\partial^3D u)\,dx
\]
\[
+\sum_{\sigma_3}\int_{\bbr^3}\partial_{x_{\sigma_3(j)}}\partial_{x_{\sigma_3(i)}}G\bkt{|Du|^2}(\partial_{x_{\sigma_3(k)}}D
u : \partial^3Du)\,dx
\]
\[
+\int_{\bbr^3}\partial^3G\bkt{|Du|^2}(D u:\partial^3D
u)\,dx=I_{31}+I_{32}+I_{33}+I_{34},
\]
where $\sigma_3=\pi_3\circ \tilde{\sigma}_3$ such that
$\tilde{\sigma}_3:\{i,j,k\}\rightarrow \{i,j,k\}$ is a permutation
of $\{i,j,k \}$ and $\pi_3$ is a mapping from $\{i,j,k\}$ to
$\{1,2,3\}$.

We separately estimate terms $I_{32}$, $I_{33}$ and $I_{34}$. We
note first that
\begin{equation}\label{eq-11 : H3}
\begin{aligned}
|I_{32}|&\leq \int_{\bbr^3}|2(G^{'}[|Du|^2]|
|Du||\partial_{x_{\sigma_3(i)}}Du||\partial_{x_{\sigma_3(k)}}\partial_{x_{\sigma_3(j)}}
D u|
|\partial^3 D u|\,dx\\
&\leq C\|G\bkt{|Du|^2}\|_{L^\infty}\|\nabla Du\|_{L^6}\|\nabla^2 Du\|_{L^3}\|\nabla^3 Du\|_{L^2}\\
&\leq C\|G\bkt{|Du|^2}\|_{L^\infty}\|\nabla^2 Du\|_{L^2} \| Du\|_{L^\infty}^{\frac{1}{3}} \|\nabla^3 Du\|^{\frac{2}{3}}_{L^2}\|\nabla^3Du\|_{L^2}\\
&\leq C\|G\bkt{|Du|^2}\|^6_{L^\infty}\| Du\|_{L^\infty}^2 \|\nabla^2
Du\|^{6}_{L^2}+\epsilon\|\nabla^3 Du\|^{2}_{L^2}.
\end{aligned}
\end{equation}
For $I_{33}$, we have
\[
|I_{33}|=\abs{\int_{\bbr^3}(
2G^{'}[|Du|^2]Du:\partial_{x_{\sigma_3(j)}}
\partial_{x_{\sigma_3(i)}} D u
+ E_2  )(\partial_{x_{\sigma_3(k)}}D u : \partial^3Du)\,dx}
\]
\[
\leq \int_{\bbr^3}2(G^{'}[|Du|^2]|Du||\nabla^2 D u| +
G\bkt{|Du|^2}|\nabla D u|^2  )|\nabla D u| |\nabla^3 Du|\,dx
\]
\[
\leq  C \|G^{'}[|Du|^2]Du\|_{L^\infty} \|\nabla^2 Du\|_{L^3}\|\nabla
Du\|_{L^6}\|\nabla^3 Du\|_{L^2} +\|G\bkt{|Du|^2}\|_{L^\infty}\|\nabla
Du\|^3_{L^6}\|\nabla^3 Du\|_{L^2}
\]
\[
\leq C\|G\bkt{|Du|^2}\|^6_{L^\infty}\|Du \|_{L^\infty}^2\|\nabla^2
Du\|^{6}_{L^2}
+C\|G\bkt{|Du|^2}\|^2_{L^\infty}\|\nabla^2Du\|^6_{L^2}+2\epsilon\|\nabla^3
Du\|^{2}_{L^2}
\]
\begin{equation}\label{eq-8 : H3}
\leq C(\|G\bkt{|Du|^2}\|^6_{L^\infty}\|Du
\|_{L^\infty}^2+\|G(|Du|)\|^2_{L^\infty})\|\nabla^2
Du\|^{6}_{L^2}+2\epsilon\|\nabla^3 Du\|^{2}_{L^2},
\end{equation}
where we use same argument as \eqref{eq-11 : H3} in the fourth
inequality. Finally, for $I_{34}$, using Lemma \ref{deriv-G}, we
note that
\begin{equation}\label{eq-5 : H3}
\begin{aligned}
I_{34}=&\int_{\bbr^3}\big( 2G^{'}[|Du|^2]Du:\partial^3 D u) + E_3\big)(D u:\partial^3D u)\,dx\\
=&2\int_{\bbr^3} G^{'}[|Du|^2]|Du:\partial^3 D u|^2\,dx +
\int_{\bbr^3} E_3(D u:\partial^3D u)\,dx.
\end{aligned}
\end{equation}
The second term in \eqref{eq-5 : H3} is estimated as follows:
\[
\int_{\bbr^3}E_3(D u:\partial^3D u)\,dx \leq  \int_{\bbr^3}|E_3||D
u||\nabla^3 D u|\,dx
\]
\[
\leq C\int_{\bbr^3} G\bkt{|Du|^2}\big(|\nabla Du|^3 +|\nabla^2
Du||\nabla Du| \big)|D u||\nabla^3 D u|\,dx
\]
\[
\leq C\|G\bkt{|Du|^2}\|_{L^\infty}\|Du\|_{L^\infty} \big( \|\nabla
Du\|_{L^6}^3 + \|\nabla Du\|_{L^6}\|\nabla^2 Du\|_{L^3} \big
)\|\nabla^3 D u\|_{L^2}
\]
\[
\leq C\|G\bkt{|Du|^2}\|^2_{L^\infty}\|Du\|^2_{L^\infty}\|\nabla^2
Du\|^6_{L^2}+C\|G\bkt{|Du|^2}\|^6_{L^\infty}\|Du\|^4_{L^\infty}\|\nabla^2
Du\|^{6}_{L^2}+2\epsilon\|\nabla^3 Du\|^{2}_{L^2}
\]
\begin{equation}\label{eq-7 : H3}
\leq
C(\|G\bkt{|Du|^2}\|^2_{L^\infty}\|Du\|^2_{L^\infty}+\|G\bkt{|Du|^2}\|^6_{L^\infty}\|Du\|^4_{L^\infty})\|\nabla^2
Du\|^{6}_{L^2}+2\epsilon\|\nabla^3 Du\|^{2}_{L^2},
\end{equation}
where we use same argument as \eqref{eq-8 : H3} in the third
inequality. Adding up the estimates \eqref{eq-1 : H3}-\eqref{eq-7 :
H3}, we obtain
\[
\frac{d}{dt}||\partial^3u||^2_{L^2(\bbr^3)}
+\int_{\bbr^3}G\bkt{|Du|^2}|\partial^3D u|^2\,dx
+\int_{\bbr^3}2G^{'}[|Du|^2]|Du:
\partial^3Du|^2\,dx
\]
\[
\leq C(\|G\bkt{|Du|^2}\|^2_{L^\infty}+  \|G\bkt{|Du|^2}\|^6_{L^\infty}
)(\|Du\|^2_{L^\infty}+\|Du\|^4_{L^\infty})\|\nabla^2
Du\|^{6}_{L^2}+5\epsilon\|\nabla^3 Du\|^{2}_{L^2}
\]
\begin{equation*}\label{eq-13 : H3}
-\int_{\bbr^3}\partial^3\big((U\cdot\nabla u) \big)\cdot
\partial^3u\,dx-  \int_{\bbr^3}\partial^2F:\partial^4 u\,dx.
\end{equation*}
Hence, we have
\[
\frac{d}{dt}||\partial^3u||^2_{L^2(\bbr^3)}
+\int_{\bbr^3}m_0|\partial^3D u|^2\,dx
\]
\[
\leq C(\|G\bkt{|Du|^2}\|^2_{L^\infty}+  \|G\bkt{|Du|^2}\|^6_{L^\infty}
)(\|Du\|^2_{L^\infty}+\|Du\|^4_{L^\infty})\|\nabla^2
Du\|^{6}_{L^2}+5\epsilon\|\nabla^3 Du\|^{2}_{L^2}
\]
\begin{equation}\label{eq-9 : H3}
-\int_{\bbr^3}\partial^3\big((U\cdot\nabla u) \big)\cdot
\partial^3u\,dx-  \int_{\bbr^3}\partial^2F:\partial^4 u\,dx.
\end{equation}
Next, we estimate the terms caused by convection terms in
\eqref{eq-3 : H1}, \eqref{eq-1 : H2} and \eqref{eq-9 : H3}.
\begin{align}\label{s-2-2}
\begin{aligned}
\sum_{1\leq |\alpha|\leq 3}\int_{\bbr^3} \partial^{\alpha}[(U\cdot
\nabla)u]\cdot \partial^{\alpha}u\,dx
&=\sum_{1\leq |\alpha|\leq 3}\int_{\bbr^3} [\partial^{\alpha}((U\cdot \nabla)u)-U\cdot \nabla \partial^{\alpha}u]\partial^{\alpha}u\,dx\\
&\leq \sum_{1\leq |\alpha|\leq3}\|\partial^{\alpha}((U\cdot \nabla)u)-U\cdot \nabla \partial^{\alpha}u\|_{L^2}\|\partial^{\alpha} u\|_{L^{2}}\\
&\leq \sum_{1\leq |\alpha|\leq3} \|\nabla U\|_{L^\infty} \|u\|_{H^{3}}\|\partial^{\alpha} u\|_{L^{2}}\\
&\leq C\|\nabla U\|_{L^\infty}\| u\|^2_{H^{3}},
\end{aligned}
\end{align}
where we use the following inequality:
\[
\sum_{|\alpha|\leq m}\int_{\bbr^3} \|\nabla^{\alpha}(
fg)-(\nabla^{\alpha}f)g\|_{L^2}\leq C(\|f\|_{H^{m-1}}\|\nabla
+\|f\|_{L^{\infty}}\|g\|_{H^m}).
\]
For the external force $F$,
\begin{align}\label{s-2-2-f}
\begin{aligned}
\sum_{0\leq |\alpha|\leq 2}\int_{\bbr^3} \partial^{\alpha}F \cdot
\partial^{\alpha}u\,dx= C\|F\|^2_{H^2}+\frac{m_0}{64}\sum_{1\leq |\beta|\leq 2}\|\partial^{\beta}u\|^2_{L^{2}}.
\end{aligned}
\end{align}
We combine \eqref{eq-1 : H0}, \eqref{eq-3 : H1}, \eqref{eq-1 : H2}
and \eqref{eq-9 : H3} with \eqref{s-2-2} and \eqref{s-2-2-f} to
conclude
\begin{equation}\label{final-est-H3-local}
\begin{aligned}
&\frac{d}{dt}||u||^2_{H^3(\bbr^3)} +\frac{m_0}{2}\int_{\bbr^3} (|\nabla^3 Du|^2+|\nabla^2 Du|^2 + |\nabla Du|^2 + | Du|^2)\,dx\\
& \leq C\|\nabla U\|_{L^\infty}\| u\|^2_{H^{3}} +C\|F\|^2_{H^{2}}
+C\|G\bkt{|Du|^2}\|_{L^\infty}(\|Du\|^2_{L^\infty}+\|Du\|_{L^\infty})\|\nabla^2
Du\|^2_{L^2}\\
&+C(\|G\bkt{|Du|^2}\|^2_{L^\infty}+\|G\bkt{|Du|^2}\|^6_{L^\infty})(\|Du\|^2_{L^\infty}+\|Du\|^4_{L^\infty})\|\nabla^2
Du\|^{6}_{L^2}.
\end{aligned}
\end{equation}
Furthermore, we have
\begin{equation}\label{eq-2 : H3 local}
\|G\bkt{|Du|^2}\|_{L^\infty} \leq \max_{0\leq s\leq
\|Du\|_{L^\infty}}G[s] \leq \max_{0\leq s\leq
C\|u\|_{H^3}}G[s]:=g(\|u\|_{H^3}),
\end{equation}
where $g:[0,\infty) \mapsto [0,\infty)$ is a nondecreasing function.
We set $X(t) := \|u(t)\|_{H^3(\bbr^3)}$ and it then follows from
\eqref{final-est-H3-local} and \eqref{eq-2 : H3 local} that
\begin{equation*}\label{eq-1 : H3 local}
\frac{d}{dt} X^2\leq f_3(X)X^2+C\|F\|^2_{H^{2}(\bbr^3)}
\end{equation*}
for some nondecreasing continuous function $f_3$, which immediately
implies that there exists $T_3>0$ such that $ \sup_{0\leq t\leq
T_3}X(t) < \infty$.

We note that $\partial_t{u}\in L^2((0,T);L^2(\bbr^3))$. Indeed, we
introduce the antiderivative of $G$, denoted by $\tilde{G}$, i.e.
$\tilde{G}[s]=\int^s_0 G[\tau]d\tau$. Multiplying $\partial_t u$ to
\eqref{NNS-F}, integrating it by parts, and using H\"{o}lder and
Young's inequalities, we have
\begin{equation}\label{reg-time-20}
\frac{1}{2}\int_{\bbr^3}|\partial_t{u}|^2\,dx
+\frac{1}{2}\frac{d}{dt}\int_{\bbr^3}\tilde{G}[|Du|^2]\,dx\leq
C\int_{\bbr^3}|\bar{u}|^2|\nabla u|^2+C\int_{\bbr^3}|F|^2\,dx.
\end{equation}
Again, integrating the estimate \eqref{reg-time-20} over the time
interval $[0,T]$, we obtain
\[
\int_0^{T}\int_{\bbr^3}|\partial_t
u|^2\,dxdt+\int_{\bbr^3}\tilde{G}[|Du(\cdot, T)|^2]\,dx
\]
\begin{equation}\label{reg-time-300}
\leq \int_{\bbr^3}\tilde{G}[|Du_0|^2]\,dx+
C\int_0^{T}\int_{\bbr^3}|\bar{u}|^2|\nabla
u|^2\,dxdt+C\int_0^{T}\int_{\bbr^3}|F|^2\,dxdt.
\end{equation}
Using Sobolev embedding, the second term in \eqref{reg-time-300} is
estimated as follows:
\[
\int_0^{T}\int_{\bbr^3}|\bar{u}|^2|\nabla u|^2\,dxdt \leq \int_0^{T}
\|u\|^2_{L^{\infty}}\|\nabla u\|^2_{L^2}dt
\]
\begin{equation}\label{convec-2}
\leq C\sup_{0<\tau\leq T}\|u(\tau)\|^2_{H^{2}}\int_0^{T} \| \nabla
u\|^2_{L^{2}}dt<\infty,
\end{equation}
and due to the assumption for the external force $F$, we also get $
\int_0^{T}\int_{\bbr^3}|F|^2\,dxdt<\infty$. Therefore, we obtain
$\partial_t{u} \in L^{2}(0,T; L^2(\bbr^3))$.

For a uniqueness of a solution, we let $u_1$ and $u_2$ be strong
solutions for the system \eqref{NNS-F}. First of all, we rewrite the
equation for $\tilde{u}:=u_1-u_2$ and $\tilde{p}:=p_1-p_2$.
\begin{equation*}
\tilde{u}_t-\nabla \cdot (G(|Du_1|^{2})-G(|Du_2|^{2})) +(U\cdot
\nabla)\tilde{u}+\nabla \tilde{p}=0,
\end{equation*}
with $\nabla \cdot \tilde{u}=0$ and $\nabla \cdot U=0$. Multiplying
$\tilde{u}$ on the both sides of the equation above and integrating
on $\bbr^3$, we obtain
\begin{align}\label{equ-uni}
\begin{aligned}
&\frac{d}{dt}\|\tilde{u}\|^2_{L^2(\bbr^3)}+m_0\|\nabla\tilde{u}\|^2_{L^2(\bbr^3)}\leq0,
\end{aligned}
\end{align}
where we use Lemma \ref{r-400} and the divergence free condition.
Applying Grownwall's inequality to estimate \eqref{equ-uni}, we get
$\tilde{u}(x,0)=0$ in $L^\infty(0,T;L^2(\bbr^3))\cap
L^2(0,T;H^{1}(\bbr^3))$. Hence this imply the uniqueness of a
solution. Finally, we introduce Galerkin approximation procedure of
the equation \eqref{NNS-F}--\eqref{NNS-F-ini} to make up
construction of solution. We omit the this part (see refer to
\cite[Proposition 3.1]{KKK2} for detailed proof).
\end{lem5}

 \section*{Acknowledgments}
Kyungkeun Kang’s work is supported by NRF-2019R1A2C1084685. Hwa Kil Kim's work is supported by NRF-2021R1F1A1048231.
Jae-Myoung Kim was supported by National Research Foundation of Korea Grant funded by the Korean Government
(NRF-2020R1C1C1A01006521).


\begin{thebibliography}{00}
\bibitem{A-F} E. Acerbi, N. Fusco, An approximation lemma for $W^{1,p}$ functions. In Material
instabilities in continuum mechanics (Edinburgh, 1985.1986), Oxford Sci. Publ.,
pages 1.5. Oxford Univ. Press, New York, 1988.

\bibitem{A-G-S} L. Ambrosio, N. Gigli, G. Savare, Gradient flows in metric spaces and the Wasserstein spaces of probability measures,
Lectures in Mathematics, ETH Zurich, Birkhauser, 2005.




\bibitem{B-C-H-K14} H.-O. Bae, Y.-P. Choi,  S.-Y. Ha,  M. Kang,
\textit{Global existence of strong solution for the
Cucker-Smale-Navier-Stokes system.} J. Differential Equations. {\bf
257}, 2225-2255 (2014).

\bibitem{B-D} C. Baranger, L. Desvillettes, Coupling Euler and Vlasov equation in the context of sprays:
the local-in-time, classical solutions. J. Hyperbolic Differ. Equ.
{\bf 3}, 1-26 (2006).

\bibitem{B-D09} L. Boudin, L. Desvillettes, C. E. Grandmont, A. Moussa, Global existence of solution for the coupled Vlasov and
Navier-Stokes equations. Differential and integral equations, {\bf
22}, 1247-1271 (2009).
%

%
%
%
%




\bibitem{C-K} Y.-P. Choi, B. Kwon, Global well-posedness and large-time behavior for the inhomogeneous
Vlasov--Navier--Stokes equations, Nonlinearity, {\bf 28} (2018)
3309--3336.

\bibitem{G-J1} T. Goudon, P.-E. Jabin, A. Vasseur, Hydrodynamic limit for the Vlasov-Navier-Stokes equations I.
Light particles regime. Indiana Univ. Math. J. {\bf 53}, 1495-1515
(2004).

\bibitem{G-J2} T. Goudon, P.-E. Jabin, A. Vasseur,
Hydrodynamic limit for the Vlasov-Navier-Stokes equations II. Fine
particles regime. Indiana Univ. Math. J. {\bf 53}, 1517-1536 (2004).


\bibitem{HKKP18} S.-Y. Ha, H. K.
Kim, J.-M. Kim, J. Park, On the global existence of weak solutions
for the Cucker-Smale-Navier-Stokes system with shear thickening. Sci
China Math, {\bf 61}, 2033-2052 (2018).

\bibitem{Ham} K. Hamdache, Global existence and large time behavior of solutions for the Vlasov-Stokes equations. Japan J. Indust. Appl. Math., {\bf 15}, 51-74 (1998).

\bibitem{HMMM17} D. Han-Kwan, E. Miot, A. Moussa, I.
Moyano, Uniqueness of the solution to the 2D Vlasov-Navier-Stokes
system, to appear, Revista Matematica Iberoamericana, 20pp.

\bibitem{KKK1} K. Kang, H. K. Kim, J. Kim,
Existence of regular solutions for a certain type of Non-Newtonian
Navier--Stokes equations, Zeitschrift fur Angewandte Mathematik und
Physik, {\bf 70}, (2019), Art. 124.


\bibitem{KKK2} K. Kang, H. K. Kim, J. Kim, Existence and temporal decay  of regular solutions to non-Newtonian fluids coupled with Maxwell
equations, Nonlinear Analysis, {\bf 180}, 284--307 (2019).

\bibitem{K-thesis} H. K. Kim, Hamiltonian systmes and the calculus of
differential forms on the Wasserstein space. Ph.D thesis, Georgia
Institute of Technology.

%

\bibitem{MPP18} P. B. Mucha, J. Peszek, M. Pokorn\'{y}, Flocking
particles in a non-Newtonian shear thickening fluid, Nonlinearity,
{\bf 31} (2018) 2703--2725.



\bibitem{Te} R. Temam, Navier-Stokes Equations, American Math. Soc., 1984




\end{thebibliography}
\end{document}